\documentclass[11pt]{article} 

\usepackage[utf8]{inputenc} 
\usepackage[T1]{fontenc}

\usepackage{geometry} 
\usepackage{graphicx,color} 

\usepackage{booktabs} 
\usepackage{array} 
\usepackage{paralist} 
\usepackage{verbatim} 
\usepackage{subfig} 
\usepackage{amsfonts, amssymb, amsmath} 
\usepackage{amsthm}
\usepackage{hyperref} 

\numberwithin{equation}{section} 

\newtheorem{theorem}{Theorem}[section]
\newtheorem{corollary}{Corollary}[theorem]
\newtheorem{lemma}[theorem]{Lemma}

\usepackage{fancyhdr} 
\usepackage{sectsty}
\allsectionsfont{\sffamily\mdseries\upshape} 

\usepackage[nottoc,notlof,notlot]{tocbibind} 
\usepackage[titles,subfigure]{tocloft} 

\allowdisplaybreaks[1]

\title{The scalar curvature flow with a flat side}
\author{Hyo Seok Jang, Ki-Ahm Lee}

\AtEndDocument{\bigskip{\footnotesize%
  \textsc{Department of Mathematical Sciences, Seoul National University, Seoul 08826, Korea} \par  
  \textit{E-mail address}, Hyo Seok Jang: \texttt{hyoseok.jang@snu.ac.kr} \par
  \addvspace{\medskipamount}
  \textsc{Department of Mathematical Sciences, Seoul National University, Seoul 08826, Korea \& Korea
Institute for Advanced Study, Seoul 02455, Korea} \par  
  \textit{E-mail address}, Ki-Ahm Lee: \texttt{kiahm@snu.ac.kr} \par
}}

\begin{document}
\maketitle
\begin{abstract}
We study the near-the-interface behavior of a compact convex scalar curvature flow with a flat side. Under suitable initial conditions on the flat side, we show that the interface propagates with a finite and non-degenerate speed until the flat side vanishes. Then we get optimal derivative estimates of the pressure-like function, optimal decay estimates of curvatures near the interface, and an Aronson-Bénilan-type curvature lower bound, from which we obtain the Hölder regularity of the ratio of the curvature to the optimal decay rate up to the free boundary. In the end, we obtain the short-time and all-time existence of the solution, smooth up to the interface. 
\end{abstract}

\section{Introduction}

\subsection{The problem of the scalar curvature flow}
This paper concerns the regularity of the free boundary problem associated with the scalar curvature flow with a flat side. We begin with the evolution equations governing the scalar curvature flow. Let a compact hypersurface \( \Sigma \) in the \( (n+1) \)-dimensional Euclidean space be given. We assume \( n \geq 3 \). Suppose that the body evolves in time by an embedding \( X:\Sigma \times [0, T) \rightarrow \mathbb{R}^{n+1} \). We denote the image of \(X \) at time \(t \) by \( M_t \). Suppose that the evolution occurs in the inward normal direction \( N=-\nu \) at each point \( x \) of the hypersurface and that the speed is given by the scalar curvature \( \sigma_2 \). Then we have the following evolution equation: 
\begin{equation} 
\frac{\partial X}{\partial t}(x,t)=-\sigma_2 (x,t) \nu (x,t) \text{  with  } X(x,0)=X_0(x) \ .
\end{equation}
A \emph{scalar curvature flow} is a solution to the equation above.  
\subsection{The history of the research on flows with a flat side} 

W. Firey \cite{firey74} first considered the evolution of the Gauss curvature flow of compact surfaces. Hamilton \cite{hamilton94} showed that if a Gauss curvature flow initially contains a flat side, then there will be a smaller flat side a little later and it takes some time for the surface to become strictly convex. 

P. Daskalopoulos with R. Hamilton \cite{daskalopoulos-hamilton99} studied the solvability and regularity of the interface \( \Gamma \) between the Gauss curvature flow and its flat side, by viewing the flow as a free boundary problem. They showed that the solution exists and is smooth up the interface, for a short time.

P.Daskalopoulos with K.-A. Lee \cite{daskalopoulos-lee03} showed the existence of regular solutions to a certain degenerate parabolic equation of the non-divergence form. Using these results of \cite{daskalopoulos-lee03} for the model equation under certain coordinates, they \cite{daskalopoulos-lee04} showed that the solution exists smoothly and the interface is smooth for all time until the flat side vanishes. 

On the other hand, P. Daskalopoulos with R. Hamilton \cite{daskalopoulos-hamilton98} studied the n-dimensional porous medium equation with a flat side. They showed the \( C^\infty \) regularity of the free boundary for a short time, using the regularity of a model degenerate equation which is obtained by changing coordinates.

P. Daskalopoulos , R. Hamilton, and K.-A. Lee showed \cite{daskalopoulos-hamilton-lee01} that the square root of the pressure is kept concave by the porous medium equation and that the solution exists for a long time and up to the free boundary while the free boundary is smooth.

K.-A. Lee and E. Lee \cite{lee-rhee12a} \cite{lee-rhee12b} considered the evolution of a rotation-invariant surface with a concave side and showed that the long-time existence of a solution, smooth up to the free boundary and the free boundary is smooth.

\subsection{The equation of the flow in the local coordinates}
We need to find the optimal regularity of the hypersurface near the free boundary where the curvatures become degenerate.

Let us assume that the embedded hypersurface of the scalar curvature flow is given as the graph of a smooth function \( y=f(x(t),t) \) for \(x \in \mathbb{R}^{n}\). Then the scalar curvature flow is given by 
\begin{equation}
\frac{\partial f}{\partial t}=\sigma_2 \sqrt{1+|\nabla_x f|^2} \ .
\end{equation}

Now, we evaluate the scalar curvature from the metric \( g_{ij} \), the second fundamental form \( h_{ij}\), and the Weingarten map \(h^i_j\), as in Ecker \cite{ecker04}.
Since
\begin{equation} 
\frac{\partial X}{\partial x^i}=\bigg(\overrightarrow e_i \ , \frac{\partial f}{\partial x^i}\bigg)\ , \ \ 1\leq i \leq n, \text{ and } \frac{\partial^2 X}{\partial x^i x^j}=\bigg(0, \frac{\partial^2 f}{\partial x^i \partial x^j}\bigg),
\end{equation}

\begin{equation} \begin{split}
g_{ij}&=\bigg<\frac{\partial X}{\partial x^i}, \frac{\partial X}{\partial x^i}\bigg>=\delta_{ij}+\frac{\partial f}{\partial x^i}\frac{\partial f}{\partial x^j}\ , \\ g^{ij}&= \bigg( \delta_{ij}-\frac{\partial_i f \partial_j f}{1+|\nabla_x f|^2}\bigg), \ \ 1\leq i, j \leq n, 
\end{split} \end{equation}

\begin{equation} \begin{split}
-\overrightarrow{\nu} &=\overrightarrow{n}=\frac{(- \nabla_x f, 1)}{\sqrt{1+|\nabla_x f|^2}} \\ 
h_{ij}&=\bigg<-\overrightarrow{\nu},\frac{\partial^2 X}{\partial x^i \partial x^j}\bigg>=\frac{1}{\sqrt{1+|\nabla_x f|^2}} \frac{\partial^2 f}{\partial x^i \partial x^j}\ , \ \ 1\leq i, j \leq n \ , \\ 
h^i_j&=g^{ik} h_{kj}=\bigg( \delta_{ik}-\frac{\partial_i f \partial_k f}{1+|\nabla_x f|^2}\bigg) \frac{1}{\sqrt{1+|\nabla_x f|^2}} \frac{\partial^2 f}{\partial x^k \partial x^j}\ , \ \ 1\leq i, j \leq n \ .
\end{split} \end{equation}

Hence, the mean curvature \(H\) of the flow is
\begin{equation} \begin{split}
H=\sum_{i=1}^nh^i_i=\sum_{i,j=1}^n g^{ij}h_{ij}=\sum_{i,j=1}^n \bigg( \delta_{ij}-\frac{\partial_i f \partial_j f}{1+|\nabla_x f|^2}\bigg) \frac{1}{\sqrt{1+|\nabla_x f|^2}} \frac{\partial^2 f}{\partial x^i \partial x^j} \ ,
\end{split} \end{equation}

and the square sum \(|A|^2=\sum_i \sigma^2_i\) of the principal curvatures is
\begin{equation} \begin{split}
&|A|^2=\sum_{i,j=1}^n h^i_j h^j_i=\sum_{i,j=1}^n g^{ik} h_{kj} g^{jl} h_{li} \\
&=\bigg( \delta_{ik}-\frac{\partial_i f \partial_k f}{1+|\nabla_x f|^2}\bigg) \frac{1}{\sqrt{1+|\nabla_x f|^2}} \frac{\partial^2 f}{\partial x^k \partial x^j} \bigg( \delta_{jl}-\frac{\partial_j f \partial_l f}{1+|\nabla_x f|^2}\bigg) \frac{1}{\sqrt{1+|\nabla_x f|^2}} \frac{\partial^2 f}{\partial x^l \partial x^i} \ ,
\end{split} \end{equation} 

so that the scalar curvature is 
\begin{equation} \label{eq:scalar}
\begin{split}
\sigma_2 & =\frac{1}{2}(H^2-|A|^2)=\frac{1}{2}\sum_{i,j,k,l=1}^n g^{ik}g^{jl}(h_{ki} h_{lj}-h_{kj} h_{li}) \\
& =\frac{1}{2(1+|\nabla_x f|^2)}\sum_{i,j,k,l=1}^n  \bigg(\delta_{ik}-\frac{\partial_i f \partial_k f}{1+|\nabla_x f|^2}\bigg)\bigg(\delta_{jl}-\frac{\partial_j f \partial_l f}{1+|\nabla_x f|^2}\bigg)\bigg(\frac{\partial^2 f}{\partial x^k \partial x^i} \frac{\partial^2 f}{\partial x^l \partial x^j} -\frac{\partial^2 f}{\partial x^k \partial x^j} \frac{\partial^2 f}{\partial x^l \partial x^i}\bigg) \\
& =\frac{1}{(1+|\nabla_x f|^2)} \sum_{i<j} \sum_{k \neq l}  \bigg(\delta_{ik}-\frac{\partial_i f \partial_k f}{1+|\nabla_x f|^2}\bigg)\bigg(\delta_{jl}-\frac{\partial_j f \partial_l f}{1+|\nabla_x f|^2}\bigg)\bigg(\frac{\partial^2 f}{\partial x^k \partial x^i} \frac{\partial^2 f}{\partial x^l \partial x^j} -\frac{\partial^2 f}{\partial x^k \partial x^j} \frac{\partial^2 f}{\partial x^l \partial x^i}\bigg) \\ 
\end{split} 
\end{equation}
which is \( \sigma_2 (h^i_j) =\sigma_2 \bigg( \bigg( \delta_{ik}-\frac{\partial_i f \partial_k f}{1+|\nabla_x f|^2}\bigg) \frac{1}{\sqrt{1+|\nabla_x f|^2}} \frac{\partial^2 f}{\partial x^k \partial x^j} \bigg) \).

From the formula of the scalar curvature (\ref{eq:scalar}), we get the following lemma.

\begin{lemma}
The scalar curvature flow is given by the graph of the function \(f\), solving
\begin{equation} \label{f_t}
\begin{split} 
f_t= \frac{1}{2\sqrt{1+|\nabla_x f|^2}}\sum_{i,j,k,l=1}^n \bigg(\delta_{ik}-\frac{f_i f_k}{1+|\nabla_x f|^2}\bigg)\bigg(\delta_{jl}-\frac{f_j f_l}{1+|\nabla_x f|^2}\bigg)(f_{ki} f_{lj}-f_{kj} f_{li})\ .
\end{split} 
\end{equation}
\end{lemma}

We note that the two-dimensional scalar curvature flow is the Gauss curvature flow:
\begin{equation} \begin{split}
f_t= \frac{\text{det} D^2 f }{(1+|\nabla_x f|^2)^{3/2}} \ .
\end{split} \end{equation}

Let us define \(g=\sqrt{2f}\) and call it the pressure-like function. Let \(I=1+g^2|\nabla_x g|^2\). Then we can express its evolution equation as the following:
\begin{equation}\begin{split} \label{eq:g_t}
g_t&=\frac{1}{2\sqrt{I}} \sum_{i,j=1}^n \big(g(g_{ii}g_{jj}-g_{ij}^2)+g_j^2 g_{ii}+g_i^2 g_{jj}-2g_i g_j g_{ij}\big) \\
&\ -\frac{1}{2I^{3/2}}g^2 \sum_{i,j,k=1}^n g_i g_k\big(g(g_{ik}g_{jj}-g_{jk}g_{ij})+g_j^2 g_{ik}+g_i g_k g_{jj}-g_i g_j g_{jk} -g_j g_k g_{ij}\big) \\
&\ -\frac{1}{2I^{3/2}}g^2 \sum_{i,j,kl=1}^n g_j g_l\big(g(g_{jl}g_{ii}-g_{il}g_{ij})+g_i^2 g_{jl}+g_j g_l g_{ii} -g_i g_j g_{il} -g_i g_l g_{ij}\big) . \\
\end{split} \end{equation} 

because the symmetry of the equation \ref{f_t} implies that
\begin{equation}
\sum_{i,j,k,l} g_i g_j g_k g_l\big((g_{ki}g_{lj}-g_{kj}g_{li})g+g_{ki}g_l g_j+g_k g_i g_{lj}-g_{kj}g_l g_i-g_k g_j g_{li}\big)=0 . \\
\end{equation}

Let \( \Gamma_\epsilon (t) \) be the level set \( \{(x,g(x,t)) | g=\epsilon \}\) and \(e_1=\nu\) be its outer normal vector. From the fact that \(g\) stays constant on the level set, we see that the equation \ref{eq:g_t} on \(\Gamma_\epsilon\) is

\begin{equation}
g_t=\frac{1}{2\sqrt{1+g^2 g_1^2}} \sum_{i,j=2}^n g(g_{ii} g_{jj}-g_{ij}^2)+\frac{1}{(1+g^2 g_1^2)^{3/2}} \sum_{i=2}^n \big(g(g_{ii}g_{11}-g_{i 1}^2)+g_1^2 g_{ii}\big).
\end{equation}

In particular, on the interface \(\Gamma(t)\) of the flat side where \(g=0\), we have
\begin{equation} 
g_t=g_\nu^2  \Delta_{\tau} g =g_\nu^3 H \ .
\end{equation} where \(H\) is the mean curvature of the interface \(\Gamma(t)\).

Also, we introduce a quantity
\begin{equation}
R_{g,2}=\sum_{i,j=1}^n (g_{ii} g_{jj}- g_{ij}^2)
\end{equation}
where \(i\) and \(j\) run through all indices \(1, \ldots, n\). The convexity of \(f\) implies that \(R_{f,2} =\sum_{i,j=1}^n (f_{ii} f_{jj}- f_{ij}^2)=g^2  R_{g,2}+2g g_1^2 \sum_{j \neq 1} g_{jj} \geq 0 \). 

\subsection{Our assumptions on the flat side and along the interface} \label{subsec:assump}
In this work, we impose the same assumptions as in P. Daskalopoulos and K.-A. Lee \cite{daskalopoulos-lee04}, the two-dimensional case. Specifically, we assume the following conditions:
\begin{itemize}
\item The hypersurface \(\Sigma\) at time \(t=0\) satisfies
\begin{equation}
\Sigma=\Sigma_0 \cup \Sigma_1
\end{equation}
where \(\Sigma_0\) is the flat side and \(\Sigma_1\) is the strictly convex part of the hypersurface. The interface between the two parts is
\begin{equation}
\Gamma=\Sigma_0 \cap \Sigma_1 .
\end{equation}

\item Because the equation (\ref{f_t}) is invariant under both rotation and translation, we can assume that \(\Sigma_0\) is in the hyperplane \( \{x_{n+1}=0\} \) and \( \Sigma_1\) lies above the hyperplane. 

\item A part of the surface \(\Sigma\), which we will also call \(\Sigma\), can be considered as the graph of a function
\begin{equation}
x_{n+1}=f(x)
\end{equation}
on a compact domain \(\Omega \subset \mathbb{R}^{n}\) containing \(\Sigma_0\), and we can set \(\Omega=\{x \in \mathbb{R}^{n} ; |Df(x)|< \infty \}\). 
\item The function f vanishes quadratically at \(\Gamma\), in other words we assume that at time \(t=0\), \(x \in \Gamma\), for any tangential direction \( \tau \) at \( \Gamma \),
\begin{equation}
|Dg(x)| \geq \lambda \ \text{ and } \ D^2_{\tau \tau} g(x) \geq \lambda
\end{equation}
for some number \(\lambda>0\). This non-degeneracy condition makes \( \Gamma \) start moving at any point at time \( t=0\) and the flow behaves as a free boundary problem.

\item A closed disc \(D_{\rho_0}=\{ x \in \mathbb{R}^n ; |x|\leq \rho_0 \} \) is contained in the flat side \(\Sigma_0 (t)\), whose area should be nonzero for time \(0 \leq t \leq T\) and \(0<T<T_c\), where \(T_c\) is the time when the area of the the flat side shrinks to zero.

\item The domain of \(f\) is contained in an \(n\)-dimensional ball \(B_R,\  R>0 \).

\end{itemize}

\subsection{Main results} 
In this paper, we prove the following theorem:
\begin{theorem}  \label{thm:main}
Under the assumptions in the subsection \ref{subsec:assump}, the pressure-like function \( g = \sqrt{2f} \) is smooth in \(\Omega(t)=\{x \in \mathbb{R}^{n} ; |Df(x,t)|< \infty \}\) up to the interface \(\Gamma(t)=\Sigma_0 (t) \cap \Sigma_1 (t) \) on time \(0 < t<T\) for all \(T<T_c\). In particular the free boundary \(\Gamma(t)\) between the strictly convex flow and the flat side will be a smooth curve for all time \(0<t<T_c\). 
\end{theorem}

\subsection{Summary}

The outline of this paper is as follows. In the section \ref{sec:speed}, we show that the interface \(\Gamma\) moves at a finite and non-degenerate speed. In the section \ref{sec:derivative} we obtain the gradient estimate of the function \(g\) and the curvature estimates. In the section \ref{sec:holder} we change the coordinates and get the Hölder regularity of the transformed function. Finally in the section \ref{sec:existence}, we prove the all-time existence and \(C^\infty\) regularity up to the interface, which is our goal.

\subsection{Notations}
Here are some notations which we will use throughout the paper:
\begin{itemize}
\item \(I=1+g^2 |\nabla g|^2\).
\item \(J=|\nabla g|^2+g\).
\item \(R_{g,2}=\sum_{i,j=1}^n (g_{ii} g_{jj}- g_{ij}^2)\).
\item \(\overline{R}_{g,2}=\sum_{i,j=2}^n (g_{ii} g_{jj} -g_{ij}^2)+\frac{2}{I}\sum_{i=2}^n (g_{ii} g_{11} -g_{1i}^2) \).
\end{itemize}

\section{Finite and non-degenerate speed of the interface of the flat side} \label{sec:speed}

This section deals with the speed of the interface of the flat side. Our goal here is to show that the speed is finite and non-degenerate. We begin with the following lemma:
\begin{lemma} \label{lem:ineq:f_e}
The scaled function with sufficiently small \(\epsilon >0\)
\begin{equation}
f_\epsilon (x,t)=\frac{1}{1+C \epsilon} f((1+A\epsilon)x, (1+B\epsilon)t)
\end{equation} is a supersolution (or subsolution, respectively) of the equation (\ref{f_t}) if 
\begin{equation} \label{ineq:f_e}
\begin{split}
B+C-4A \ & \geq (C-A)\left( \frac{5|\nabla f_\epsilon|^2}{1+|\nabla f_\epsilon|^2}-2\frac{|\nabla f_\epsilon|^2 \delta_{ik}-f_{\epsilon,i} f_{\epsilon,k} }{(1+|\nabla f_\epsilon|^2) \delta_{ik}-f_{\epsilon,i} f_{\epsilon,k} } -2\frac{|\nabla f_\epsilon|^2 \delta_{jl}-f_{\epsilon,j} f_{\epsilon,l} }{(1+|\nabla f_\epsilon|^2) \delta_{jl}-f_{\epsilon,j} f_{\epsilon,l} } \right) \\
\\ & (\text{or} \ \leq , \ \text{respectively})
\end{split} 
\end{equation}
for each \( 1\leq i,j,k,l \leq n \).
\end{lemma}

\begin{proof}
Let \(A\), \(B\), \(C \in \mathbb{R}\) be constants. We scale the function \(f(x,t)\) with those constants by
\begin{equation}
f_\epsilon (x,t)=\frac{1}{1+C \epsilon} f((1+A\epsilon)x, (1+B\epsilon)t) \ , \ \epsilon > 0 \ .
\end{equation}

Then its partial derivatives of the first order are
\begin{equation}
f_{\epsilon,t}(x,t)=\frac{1+B\epsilon}{1+C\epsilon}f_t((1+A\epsilon)x, (1+B\epsilon)t),
\end{equation}
\begin{equation} 
f_{\epsilon,i}=\frac{\partial f_\epsilon}{\partial x^i}=\frac{1+A\epsilon}{1+C\epsilon}f_i((1+A\epsilon)x, (1+B\epsilon)t),
\end{equation} with
\begin{equation} \label{grad f_e}
|\nabla f_\epsilon|^2=\left( \frac{1+A\epsilon}{1+C\epsilon} \right)^2 | \nabla f((1+A\epsilon)x, (1+B\epsilon)t)|^2
\end{equation}

We denote \(M=\left(\frac{1+C\epsilon}{1+A\epsilon} \right)^2\). Rearranging the equation of the scalar curvature flow about \(f_\epsilon\), we get the following equation:
\begin{equation} \begin{split} \label{eq: f_e,t}
\frac{\partial}{\partial_t} f_{\epsilon} &= \frac{E_{ijkl} (x)}{2\sqrt{1+|\nabla_x f|^2}}(\delta_{ik}-\frac{f_i f_k}{1+|\nabla_x f|^2})(\delta_{jl}-\frac{f_j f_l}{1+|\nabla_x f|^2})(f_{ki} f_{lj}-f_{kj} f_{li})  \ ,
\end{split} \end{equation}
where
\begin{equation} \begin{split}
E_{ijkl} (x) =&\frac{(1+B\epsilon)(1+C\epsilon)}{(1+A\epsilon)^4} \left( \frac{1+|\nabla f_\epsilon|^2}{1+M|\nabla f_\epsilon|^2} \right)^{5/2}   \frac{(1+M |\nabla f_\epsilon|^2) \delta_{ik}-M f_{\epsilon,i} f_{\epsilon,k} }{(1+|\nabla f_\epsilon|^2) \delta_{ik}-f_{\epsilon,i} f_{\epsilon,k} } \frac{(1+M |\nabla f_\epsilon|^2) \delta_{jl}-M f_{\epsilon,j} f_{\epsilon,l} }{(1+ |\nabla f_\epsilon|^2) \delta_{jl}- f_{\epsilon,j} f_{\epsilon,l} } .
\end{split} \end{equation}

When \(\epsilon \ll 1\), the factor \(E_{ijkl} (x)\) is
\begin{equation} \begin{split}
&E_{ijkl} (x)= (1+(B+C-4A)\epsilon)( 1-5(C-A)\epsilon \frac{|\nabla f_\epsilon|^2}{1+|\nabla f_\epsilon|^2} ) \left( 1+2(C-A)\epsilon \frac{|\nabla f_\epsilon|^2 \delta_{ik}-f_{\epsilon,i} f_{\epsilon,k} }{(1+|\nabla f_\epsilon|^2) \delta_{ik}-f_{\epsilon,i} f_{\epsilon,k} } \right) \\
& \cdot \left( 1+2(C-A)\epsilon \frac{|\nabla f_\epsilon|^2 \delta_{jl}-f_{\epsilon,j} f_{\epsilon,l} }{(1+|\nabla f_\epsilon|^2) \delta_{jl}-f_{\epsilon,j} f_{\epsilon,l} } \right) +o(\epsilon^2) \\
&= 1+\epsilon \left( B+C-4A-(C-A)\left( \frac{5|\nabla f_\epsilon|^2}{1+|\nabla f_\epsilon|^2}-2\frac{|\nabla f_\epsilon|^2 \delta_{ik}-f_{\epsilon,i} f_{\epsilon,k} }{(1+|\nabla f_\epsilon|^2) \delta_{ik}-f_{\epsilon,i} f_{\epsilon,k} } -2\frac{|\nabla f_\epsilon|^2 \delta_{jl}-f_{\epsilon,j} f_{\epsilon,l} }{(1+|\nabla f_\epsilon|^2) \delta_{jl}-f_{\epsilon,j} f_{\epsilon,l} } \right) \right) \\
&\ +o(\epsilon^2) \\
& \geq 1 (\text{or}  \leq 1 ),
\end{split} \end{equation}
if the inequality with \( \geq (\text{or}  \leq , \ \text{respectively} ) \) in \( (\ref{ineq:f_e}) \) holds for each \( 1\leq i,j,k,l \leq n \).

So the evolution equation (\ref{eq: f_e,t}) of the scaled function \(f_{\epsilon}(x,t) \) finishes the proof when \(\epsilon>0\) is sufficiently small.
\end{proof}

This lemma implies the following lemma, from which we see that the propagating speed of the free boundary is finite.
\begin{lemma} \label{ineq:f_e}
There exist a constant \(\delta_0>0\), a negative constant \(B\) and a positive constant \(C\) satisfying
\begin{equation}
-C f(x,t) + x \cdot \nabla f(x,t)+Bt f_t(x,t) \geq 0
\end{equation} on the set \(\mathcal{A}_{\delta_0}=\{(x,t); \ 0<f(x,t) \leq \delta_0, \ 0\leq t \leq T \}\).
\end{lemma}

\begin{proof}
Because \( f(x,t) \) is uniformly continuous on \( t \in [0,T] \), for any \(0< \eta <1\) there exists \(0<\delta_0 \ll \eta < 1\) such that 
\begin{equation} \label{set:A_D}
\{ x ; 0 < f(x,t) \leq \delta_0 \} \subset \{ x ; d(x,\Gamma(t)) \leq \frac{\eta \rho_0}{2} \} \ ,
\end{equation}
\begin{equation} \label{set:near G}
\{ (1+\epsilon) x ; d(x,\Gamma(t)) \leq \frac{\eta \rho_0}{2} \} \subset \{ x ; d(x,\Gamma(t)) \leq \eta \rho_0 \}
\end{equation}
on \(0\leq t \leq T\), for all \(\epsilon \ll \delta_0\). Let \(\Gamma(t)=\partial \{x ; f(x,t)>0 \}\) be the interface.
Consider the scaled function \(f_\epsilon\) as in Lemma \ref{lem:ineq:f_e} with \(A=1\), \(B=-\delta_0^2\), and \(C=8\). To prove this lemma, it is sufficient to show that \(f_\epsilon \geq f\) on \(\mathcal{A}_{\delta_0}\) so that \(\frac{d}{d \epsilon} |_{\epsilon=0} f_\epsilon \geq 0\) on \(\mathcal{A}_{\delta_0}\).\\

The inequality (\ref{ineq:f_e}) in Lemma \ref{lem:ineq:f_e} for \(f_\epsilon\) to be a supersolution becomes
\begin{equation}
-\delta_0^2+4 \geq 7 \left( \frac{5|\nabla f_\epsilon|^2}{1+|\nabla f_\epsilon|^2}-2\frac{|\nabla f_\epsilon|^2 \delta_{ik}-f_{\epsilon,i} f_{\epsilon,k} }{(1+|\nabla f_\epsilon|^2) \delta_{ik}-f_{\epsilon,i} f_{\epsilon,k} } -2\frac{|\nabla f_\epsilon|^2 \delta_{jl}-f_{\epsilon,j} f_{\epsilon,l} }{(1+|\nabla f_\epsilon|^2) \delta_{jl}-f_{\epsilon,j} f_{\epsilon,l} } \right) 
\end{equation}
for each \( 1\leq i,j,k,l \leq n \).

We can observe that if \( i \neq j \)
\begin{equation}
\frac{|\nabla f_\epsilon|^2 \delta_{ij}-f_{\epsilon,i} f_{\epsilon,j} }{(1+|\nabla f_\epsilon|^2) \delta_{ij}-f_{\epsilon,i} f_{\epsilon,j} }=1 \ ,
\end{equation} and that if \( i=j \)
\begin{equation}
0 \leq \frac{|\nabla f_\epsilon|^2  }{1+|\nabla f_\epsilon|^2 } \leq \frac{|\nabla f_\epsilon|^2 \delta_{ij}-f_{\epsilon,i} f_{\epsilon,j} }{(1+|\nabla f_\epsilon|^2) \delta_{ij}-f_{\epsilon,i} f_{\epsilon,j} }=\frac{ \sum_{j \neq i} f_{\epsilon,j}^2 }{ 1+\sum_{j \neq i} f_{\epsilon,j}^2 } <1,
\end{equation}
so we have 
\begin{equation}
\begin{split}
& -\delta_0^2+4 - 7 \left( \frac{5|\nabla f_\epsilon|^2}{1+|\nabla f_\epsilon|^2}-2\frac{|\nabla f_\epsilon|^2 \delta_{ik}-f_{\epsilon,i} f_{\epsilon,k} }{(1+|\nabla f_\epsilon|^2) \delta_{ik}-f_{\epsilon,i} f_{\epsilon,k} } -2\frac{|\nabla f_\epsilon|^2 \delta_{jl}-f_{\epsilon,j} f_{\epsilon,l} }{(1+|\nabla f_\epsilon|^2) \delta_{jl}-f_{\epsilon,j} f_{\epsilon,l} } \right) \\
& \geq -\delta_0^2+4 - 7 \frac{|\nabla f_\epsilon|^2}{1+|\nabla f_\epsilon|^2}
\end{split}
\end{equation}
for each \( 1\leq i,j,k,l \leq n \). This inequality holds if \( |\nabla f_\epsilon|^2 \leq \frac{1}{6} \) and \( 0<\delta_0 <1\) . By the equation (\ref{grad f_e}), \( |\nabla f_\epsilon (x,t) | \leq |\nabla f((1+\epsilon)x, (1+B\epsilon)t )| \) with \(A=1\), \(B=-\delta_0^2\), and \(C=8\). So it is enough to show that  \( |\nabla f((1+\epsilon)x, (1-\delta_0^2 \epsilon)t )|^2 \leq \frac{1}{ \sqrt{6} } \) .

We note that the function \(f\) is of the class \(C^{1,1}\). When the dimension \(n\) of the scalar curvature flow is 2, the flow becomes the Gauss curvature flow which has been considered by \cite{daskalopoulos-lee04}; they used the result in Andrews \cite{andrews99} that the viscosity solution of the scalar curvature flow has bounded mean curvature \(H\). Hence, with the convexity assumption, all principal curvatures are positive and bounded by a uniform constant and the function \(f\) should be \(C^{1,1}\) as in \cite{andrews99}. \\

For scalar curvature flows of higher dimension \(n >2 \), the sum of all the squared principal curvatures \(|A|^2= \sum_{i=1}^n \lambda_i^2 \) and the mean curvature \(H\) of the flow are all bounded. This proves that every viscosity solution \(f\) to the scalar curvature flow has a uniform \(C^{1,1}\) estimate. As \(|\nabla f|=0\) along the interface \(\Gamma(t)\), the uniform \(C^{1,1}\) estimate of \(f\) implies that, for some positive constant \(M\) which does not depend on time \(t\), we have for \((x,t) \in \mathcal{A}_{\delta_0}\)
\begin{equation} \label{bdd:Df<d}
|\nabla f(x,t)|\leq M d(x, \Gamma(t)) .
\end{equation}

Note that the function \(f\) is nondecreasing in time \(t\) and \( f_t \geq 0\), because the scalar curvature \(\sigma_2\) is positive for all \(t>0\). The positivity of \(\sigma_2\) is derived from the fact that if the scalar flow is convex at \(t=0\), then it remains convex for all \(t>0\). \\

Assume that \( (x,t) \in \mathcal{A}_{\delta_0} \), i.e. \( f(x,t) \leq \delta_0 \). Then \( f((1+\epsilon)x, (1-\delta_0^2 \epsilon)t ) \leq \delta_0 \) because \( f_t \geq 0\). By (\ref{set:A_D}), we have
\begin{equation}
d(x, \Gamma((1-\delta_0^2 \epsilon)t) ) \leq \frac{\eta \rho_0}{2},
\end{equation}
and by (\ref{set:near G}),
\begin{equation}
d((1+\epsilon), \Gamma((1-\delta_0^2 \epsilon)t) ) \leq \eta \rho_0
\end{equation}
whose right-hand side becomes arbitrarily small if we choose \( \eta \) small. Consequently, by (\ref{bdd:Df<d}), we can make \( |\nabla f((1+\epsilon)x, (1-\delta_0^2 \epsilon)t )|^2 \leq \frac{1}{ \sqrt{6} } \) .

By applying the comparison principle to the supersolutions \(f_\epsilon\) and \(f\), we can show that if additionally \(f_\epsilon \geq f\) on the parabolic boundary \(\partial_p \mathcal{A}_{\delta_0}=\{(x,t) ; f(x,t)= \delta_0, 0 \leq t \leq T \}\) of \(\mathcal{A}_{\delta_0}\), then \(f_\epsilon \geq f\) in \(\mathcal{A}_{\delta_0}\). Now, the only remaining part is to prove that \(f_\epsilon \geq f\) on \(\partial_p \mathcal{A}_{\delta_0}\). From the simple calculation
\begin{equation} \begin{split}
\frac{d}{d \epsilon}f_\epsilon (x,t)=&-\frac{C}{(1+C \epsilon)^2} f((1+A\epsilon)x, (1+B\epsilon)t) +\frac{Ax}{1+C \epsilon}\cdot \nabla_x   f((1+A\epsilon)x, (1+B\epsilon)t) \\
&+\frac{Bt}{1+C \epsilon} f_t ((1+A\epsilon)x, (1+B\epsilon)t)  , 
\end{split}\end{equation} we have, for \(A=1\) in our assumption,
\begin{equation} \begin{split}
\frac{d}{d \epsilon}f_\epsilon (x,0)|_{\epsilon=0}=&-C f(x, 0) +x\cdot \nabla_x   f(x, 0) .
\end{split}\end{equation} \\

As \(x\cdot \nabla_x   f(x, 0)>0\), for small enough \(\delta_0\), \(\frac{d}{d \epsilon}f_\epsilon (x,0)|_{\epsilon=0}>0\) on \(\{x; \ f(x,0) \leq \delta_0\}\) so that \(f_\epsilon (x,0) \geq f(x,0)\) for small \(\epsilon >0\). When \(f(x,t)=\delta_0\), we see that \(d(x,\Gamma(t)) \leq \eta \rho_0 /2 \leq \eta \rho_0\). Since \(f\) is convex in the radius \(r\), the radial derivative \(f_r\) satisfies \(f_r \geq \frac{\delta_0}{\eta \rho_0}\) and it holds that \(r=d(0,x) \geq \rho_0 \) and \(x \cdot \nabla_x f = r f_r (x,t) \geq \rho_0 \frac{\delta_0}{\eta \rho_0}=\frac{\delta_0}{\eta}\) on \(\partial_p \mathcal{A}_{\delta_0}\). \\

Hence, for \(f=\delta_0\) on \(\partial_p \mathcal{A}_{\delta_0}\), if \(\eta\) is small enough,
\begin{equation}
\frac{d}{d \epsilon}f_\epsilon (x,t)|_{\epsilon=0}=-C f(x, t) +x\cdot \nabla_x   f(x, t)+Bt f_t (x, t) \geq -C \delta_0 +\frac{\delta_0}{\eta}-\delta_0^2 T |f_t|_{L^\infty}>0 ,
\end{equation} 
which means that \(f_\epsilon \geq f\) on \(\partial_p \mathcal{A}_{\delta_0}\), and the comparison principle finishes the proof.
\end{proof}

From the proof of Lemma \ref{ineq:f_e}, we see that \(f \in C^{1,1}\). Thus we may assume that  
\begin{equation}
\Omega(t)=\{x \in \mathbb{R}^{n} ; |Df(x,t)|< \infty \}
\end{equation}
for all time \(0<t<T_c\).

We can show that the free boundary moves with a finite speed when \(0 \leq t \leq T\). The radius of the free boundary \(\Gamma (t) \) is written by \(r=\gamma(\theta,t)\) with \(0 \leq \theta < 2\pi\)  in polar coordinates.
\begin{theorem}
Under the assumption that \(D_{\rho_0} \subset \Sigma_0 (T)\), there is a constant \(B>0\) such that
\begin{equation}
\gamma(\theta,t) \geq e^{-\frac{t-t_0}{B t_0}} \gamma(\theta, t_0)
\end{equation} for all \(0<t_0\leq t \leq T\) and \(0\leq \theta<2\pi\). In particular, the free boundary moves with a finite speed when \(0 \leq t \leq T\).
\end{theorem}

\begin{proof}
From Lemma \ref{ineq:f_e}, for \(0<t_0 \leq t \leq T\), we get the inequality
\begin{equation}\begin{split}
0 \geq \frac{C f(x,t)}{B t} - \frac{x}{ B t} \cdot \nabla f(x,t)+f_t(x,t)  \geq \frac{C f(x,t)}{ B T} - \frac{x}{ B t_0} \cdot \nabla f(x,t)+f_t(x,t) 
\end{split}\end{equation} 
so that
\begin{equation}\begin{split}
&\frac{d}{dt}(e^{\frac{C}{ B T}(t-t_0)}f( e^{-\frac{t-t_0}{ B t_0}} x, t))\\
&=e^{\frac{C}{ B T}(t-t_0)}\left( \frac{C f( e^{-\frac{t-t_0}{ B t_0}} x, t)}{ B T} - \frac{x}{ B t_0} \cdot \nabla f( e^{-\frac{t-t_0}{ B t_0}} x, t)+f_t( e^{-\frac{t-t_0}{ B t_0}} x, t) \right) \leq 0
\end{split}\end{equation} and hence
\begin{equation}\begin{split} \label{ineq:exp-f}
 e^{\frac{C}{ B T}(t-t_0)}f( e^{-\frac{t-t_0}{ B t_0}} x, t)\leq f(x,t_0) =0 \\
\end{split}\end{equation}
for \(|x|=\gamma(x,t_0)\), which implies the conclusion.

\end{proof}

For small \(\epsilon>0\) the \(\epsilon\)-level set of the function f moves with a finite speed as well, by the following theorem. Let us express the \(\epsilon\)-level set by its radius \(r=\gamma_\epsilon (\theta,t)\).

\begin{theorem}
Under the assumption that \(D_{\rho_0} \subset \Sigma_0 (T)\), there is a constant \(B<0\) such that
\begin{equation}
\gamma_\epsilon (\theta,t) \geq e^{-\frac{t-t_0}{|B|t_0}} \gamma_\epsilon (\theta, t_0)
\end{equation} for all sufficiently small \(\epsilon>0\), \(0<t_0\leq t \leq T\), and \(0\leq \theta<2\pi\). In particular, for each \(\epsilon>0\) the level set \(r=\gamma_\epsilon (\theta, t)\) moves with a finite speed when \(0 \leq t \leq T\). 
\end{theorem}

\begin{proof}
Fix \(0 \leq \theta < 2\pi\). Let \(r_0\) be \(r_0=\gamma_\epsilon (\theta, t_0)\) and \(x_0\) be the point \((r_0,\theta)\) in polar coordinates. Then \(f(x_0,t_0)=\epsilon\) and by the inequality (\ref{ineq:exp-f}), we have
\begin{equation}
f( e^{-\frac{t-t_0}{|B|t_0}} x_0, t) \leq e^{-\frac{C}{|B|T}(t-t_0)} f( x_0, t_0) \leq f( x_0, t_0) = \epsilon=f(\gamma_\epsilon(\theta,t_0)),
\end{equation}
implying that
\begin{equation}
e^{-\frac{t-t_0}{|B|t_0}}\gamma_\epsilon (\theta, t_0) \leq \gamma_\epsilon(\theta,t) \ .
\end{equation}
\end{proof}

\begin{lemma}
There exist constants \(A>0\), \(B<0\), \(C>0\), and \(D>0\) such that
\begin{equation} \label{ineq:f_x f_t}
-Cf(x,t)+Ax \cdot \nabla_x f(x,t)+(-D+Bt) f_t (x,t)  \leq 0
\end{equation}
on \(\{f(x,t) \leq 1, \ \ 0\leq t \leq T\}\).
\end{lemma}

\begin{proof}
Assume that \(g\) is smooth up to the interface \(\Gamma(t)\) for time \(0\leq t \leq \tau\) for some \(\tau >0\). Let \(t^*=\tau/2\) and \(\mathcal{A}_{t^*}=\{ f(x,t) \leq 1, t^* \leq t \leq T\}\). We want to show that for some negative constants \(A\) and \(B\), and positive constants \(C\) and \(D\), and for sufficiently small \(\epsilon>0\),
\begin{equation}
f_\epsilon (x,t)=\frac{1}{1+ C\epsilon} f((1+A\epsilon)x, (1+B\epsilon)t-D\epsilon) \leq f(x,t)
\end{equation} on \(\mathcal{A}_{t^*}\). We first choose \(C=1\) so that by Lemma \ref{lem:ineq:f_e}, \(f_\epsilon\) is a subsolution to the equation for the scalar curvature flow if and only if \(B+1-4A \leq 0\). Given \(A>0\) to be determined later, take \(B<0\) such that \(B\leq 4A-1\) . Then \(f_\epsilon\) is a subsolution, especially in \(\mathcal{A}_{t^*}\). Therefore, the comparison principle implies that it suffices to show that \(f_\epsilon \leq f\) on the parabolic boundary of \(\mathcal{A}_{t^*}\), where \(f(x,t)=1, \ t^*\leq t \leq T\) or \(f(x,t)\leq 1, \ t=t^*\). It is equivalent to show that \(\frac{\partial f_\epsilon}{\partial \epsilon}|_{\epsilon=0}=-f(x,t)+Ax \cdot \nabla_x f(x,t)+(-D+Bt) f_t (x,t)  \leq 0\). \\

On \(\{f(x,t)=1, \ t^*\leq t \leq T\}\), we can take \(A>0\) sufficiently small that we have \(-f(x,t)+Ax \cdot \nabla_x f(x,t)+(-D+Bt) f_t (x,t)=-1+Ax \cdot \nabla_x f(x,t) \leq 0\) since \(x \cdot \nabla_x f(x,t) \geq 0\) and \(f(\cdot,t) \) is uniformly \(C^{1,1}\) in \(0 \leq t \leq T \), \(f \leq 1\). Secondly, on the set \(\{f(x,t)\leq 1, \ t=t^*\}\) and for the pressure-like function \(g=\sqrt{2f}\), we see that \(-f(x,t)+Ax \cdot \nabla_x f(x,t)+(-D+Bt) f_t (x,t)=g\{-\frac{1}{2} g(x,t^*)+Ax\cdot \nabla_x g(x,t^*)+(-D+B t^*) g_t (x,t^*)\}\). Because \(g(x,t^*)>0\), \(B<0\), and \(g_t(x,t^*) \geq 0\), we only need to show that \(-g(x,t^*)+2Ax\cdot \nabla_x g(x,t^*)-2D g_t (x,t^*) \leq 0\). \\

However, we assumed the initial non-degeneracy condition for \(f\) and hence for \(g\) so that for sufficiently small \(t^*\)
\begin{equation}
|\nabla_x g (x,t^*)| \geq c \ \ \text{and} \ \ g_{ \tau \tau} \geq c \ \ \text{at the interface} \ \ \Gamma(t^*)
\end{equation} for some \(c>0\). For this reason there are some \(\rho>0\) and \(c_0>0\) satisfying \(g_t \geq c_0 >0\),  so that \(-g(x,t^*)+2A x\cdot \nabla_x g(x,t^*)-2D g_t(x,t^*)\leq0\) on the set \(\{0<g(x,t^*)<\sqrt{2}, \ d(x,\Gamma(t^*) < \rho \}\) if we take sufficiently small \(A>0\) since \(D>0\). The same estimate holds when \(0\leq t \leq t^*\). Hence \(f_\epsilon \leq f \) on \(\{ f(x,t) \leq 1, 0\leq t \leq T\}\).
\end{proof}

\begin{theorem}
There exists constants \(A>0\), \(B<0\), and \(D>0\) such that
\begin{equation}
\gamma(\theta,t) \leq e^{-\frac{A(t-t_0)}{D+|B|T}} \gamma (\theta, t_0)
\end{equation} for \(0<t_0\leq t<T\). Hence, the interface moves with a non-degenerate speed.
\end{theorem}
\begin{proof}
For time \(0<t_0<t<T\), the inequality (\ref{ineq:f_x f_t}) implies that
\begin{equation} \begin{split}
0 &\leq \frac{C f(x,t)}{D+|B|t}-\frac{Ax}{D+|B|t}\cdot \nabla_x f(x,t) +f_t(x,t) \\ 
&\leq \frac{C f(x,t)}{D+|B|t_0}-\frac{Ax}{D+|B|T}\cdot \nabla_x f(x,t) +f_t(x,t)
\end{split} \end{equation} by which we obtain
\begin{equation} \begin{split}
\frac{d}{dt}(e^{\frac{C f(x,t)}{D+|B|t_0}}f(e^{-\frac{A(t-t_0)}{D+|B|T}}x,t)) \geq 0 \ .
\end{split} \end{equation}
Hence we have
\begin{equation}
e^{\frac{C f(x,t)}{D+|B|t_0}}f(e^{-\frac{A(t-t_0)}{D+|B|T}}x,t) \geq f(x,t_0) \ ,
\end{equation} and for sufficiently large \(D>0\)
\begin{equation}
f(e^{-\frac{A(t-t_0)}{D+|B|T}}x,t) \geq f(x,t_0) \ .
\end{equation} 
From the last inequality and the monotonicity of \(f\) in the radius, the conclusion follows.
\end{proof}

\section{Derivative estimates} \label{sec:derivative}
\subsection{Evolution equations of derivatives of \(g\)}
Let us consider the following linear operator, \(L[w]\),  which will occur in the evolution equations of \(g_m\) and \(g_{mp}\):
\begin{equation}
L[w]=\sum_{i,j=1}^n a_{ij}w_{ij} +\sum_{i=1}^n b_i w_i,
\end{equation}
where the coefficients are defined by
\begin{equation}
\begin{split}
a_{ii} &=\frac{1}{I^{3/2} } \Big(\sum_{j \neq i} \big( g_j^2 +g g_{jj}+g^3 ( g_k^2 g_{jj}-g_j g_kg_{kj} )  \big) -g^3 (g_i^2 g_{kk}-g_i g_k g_{ik} ) \Big) ,\\
a_{ij}&= -\frac{1}{I^{3/2} } \big( g_i g_j+g g_{ij}+g^3 (g_k^2 g_{ij}+ g_i g_jg_{kk}-g_i g_k g_{jk} -g_j g_k g_{ik} ) \big) , \text{ for } i \neq j ,\\
\end{split}
\end{equation}
and 
\begin{equation}
\begin{split}
b_i&=\frac{1}{2I^{5/2}} \Big( 4I ( g_i g_{jj}-g_j g_{ij}) +6g^5 g_i g_k g_l (g_{kl}g_{jj}-g_{jk}g_{jl}) \\
&\ \ \ +g^2 \big(- 6g_i  (g_k^2 g_{jj}-g_k g_j g_{kj}) -4I g g_k (g_{ik}g_{jj}-g_{jk}g_{ij}) -I g g_i (g_{kk}g_{jj}-g_{kj}^2)  \big)  \Big) .\\
\end{split}
\end{equation}

Also, we define
\begin{equation}
\begin{split}
c&=\frac{1}{2I^{5/2}} \big( I ( g_{ii}g_{jj}-g_{ij}^2)-6g^2 g_i g_k (g_{ik}g_{jj}-g_{jk}g_{ij}) - 6g g_i^2 ( g_j^2 g_{kk}- g_k g_j g_{kj}  ) \big) . \\
\end{split}
\end{equation}

\begin{lemma} \label{lem:g_m}
For \(1\leq m \leq n\), \((g_{m})_t\) is given by
\begin{equation} \label{eq:g_m}
\begin{split}
(g_{m})_t&=L[g_m] +c g_m .\\
\end{split}
\end{equation}
\end{lemma}

\begin{proof}
This equation (\ref{eq:g_m}) can be obtained by a direct calculation.

\end{proof}

In the local coordinates where \(g_1=g_\nu\), \(g_i=0\) for \(i \neq 1\) and \(g_{ij}=0\) for \(i \neq j\), we have
\begin{equation}
\begin{split}
a_{11} &=\frac{1}{I^{3/2} } g \sum_{j \neq 1} g_{jj} ,\ \ a_{ij}= 0 , \text{ for } i \neq j ,\\
a_{ii} &=\frac{1}{I^{3/2} } \Big( g_1^2+g g_{11} +Ig \sum_{j \neq 1, i} g_{jj}  \Big) , \text{ for } i \neq 1, \\
\end{split}
\end{equation} 

\begin{equation}
\begin{split}
b_1&=\frac{1}{2I^{5/2}} \Big( -I g^3 g_1 R_{g,2}+\big(4I g_1  -6g^2 g_1 (g_1^2+g g_{11}) +2I g^3 g_1 g_{11}\big)\sum_{j \neq 1} g_{jj}   \Big) ,\\
b_i&=0, \text{ for } i \neq 0, \\
\end{split}
\end{equation}
and
\begin{equation}
\begin{split}
c&=\frac{1}{2I^{5/2}} \big( IR_{g,2} -6g g_1^2 (g_1^2 +g g_{11})\sum_{ j \neq 1 } g_{jj}  \big) , \\
\end{split}
\end{equation}
simplifying the evolution of  \(g_m\) as follows:

\begin{equation}
\begin{split}
&(g_{1})_t=\frac{1}{I^{3/2} } g \sum_{j \neq 1} g_{jj} g_{111} +\sum_{i \neq 1} \frac{1}{I^{3/2} } \Big( g_1^2+g g_{11} +Ig \sum_{j \neq 1, i} g_{jj}  \Big) g_{1ii} \\
&\ \ +\frac{1}{2I^{5/2}} \Big( -I g^3 g_1 R_{g,2}+\big(4I g_1  -6g^2 g_1 (g_1^2+g g_{11}) +2I g^3 g_1 g_{11}\big)\sum_{j \neq 1} g_{jj}   \Big) g_{11}\\
&\ \ +\frac{1}{2I^{5/2}} \big( IR_{g,2} -6g g_1^2 (g_1^2 +g g_{11})\sum_{ j \neq 1 } g_{jj}  \big) g_1  ,\\
\end{split}
\end{equation}
and for \(m \neq 1\)
\begin{equation}
\begin{split}
&(g_{m})_t=\frac{1}{I^{3/2} } g \sum_{j \neq 1} g_{jj} g_{m11} +\sum_{i \neq 1} \frac{1}{I^{3/2} } \Big( g_1^2+g g_{11} +Ig \sum_{j \neq 1, i} g_{jj}  \Big) g_{mii}   .\\
\end{split}
\end{equation}

\begin{lemma}
For \(1\leq m, p \leq n\), \((g_{mp})_t\) is given by
\begin{equation} \label{eq:g_mp}
\begin{split}
(g_{mp})_t&=\sum_{i,j=1}^n a_{ij}g_{mpij} +\sum_{i,j,k,l=1}^n b_{mp,ij,kl}g_{mij}g_{pkl} +\sum_{i,j,k=1}^n c_{mp,ijk}g_{ijk}+d_{mp} , \\
\end{split}
\end{equation}
with the fourth derivatives

\begin{align*}
\sum_{i,j=1}^n a_{ij}g_{mpij}&=\frac{1}{I^{3/2}} \sum_{i,j=1}^n ( Ig g_{jj}-g^3 g_j \sum_{k=1}^n g_k g_{jk} +g_j^2 )g_{mpii} \\
&\ -\frac{1}{I^{3/2}} \sum_{i,j=1}^n ( I g g_{ij}+ g^3 g_i g_j \sum_{k=1}^n g_{kk} -2g^3  g_i \sum_{k=1}^n g_k g_{jk} +g_i g_j ) g_{mpij}  , \\
\end{align*}
the terms which are quadratic in third derivatives

\begin{align*}
\sum_{i,j,k,l=1}^n b_{mp,ij,kl}g_{mij}g_{pkl}&=\frac{1}{I^{3/2}}  \sum_{i=1}^n g_{mii} ( Ig \sum_{k=1}^n g_{pkk}-g^3 \sum_{k,l=1}^n g_k g_l g_{pkl}) -\frac{1}{I^{3/2}} I g \sum_{i,j=1}^n g_{mij} g_{pij} \\
&\ +\frac{1}{I^{3/2}} g^3 \sum_{i,j,k=1}^n g_k  g_{mij} (g_j g_{pik}+g_i g_{pjk}) -\frac{1}{I^{3/2}} g^3 \sum_{i,j=1}^n g_i g_j g_{mij} \sum_{k=1}^n g_{pkk} , 
\end{align*}
the terms which are linear in third derivatives

\begin{align*}
&\sum_{i,j,k=1}^n c_{mp,ijk}g_{ijk}=\frac{1}{2I^{5/2}} \sum_{i=1}^n \Big( 4I\sum_{j=1}^n (g_i g_{jj} -g_j  g_{ij} ) -Ig^2 \sum_{j,k=1}^n \big( gg_i  (g_{kk}g_{jj}-g_{kj}^2) +4gg_k (g_{ik}g_{jj}-g_{jk}g_{ij})\big) \\
&\ \ -\sum_{j,k=1}^n g^2(6g_i g_k^2 g_{jj}-6g_i g_j g_k g_{jk} ) +6g^5 \sum_{j,k,l=1}^n ( g_i g_l g_k g_{lk}g_{jj}-g_i g_l g_k g_{jl}g_{jk})  \Big) g_{mpi} \\
&\ -\frac{3}{I^{5/2}} g^2 \sum_{j=1}^n ( Ig g_{jj} +g_j^2 -g^3 g_j \sum_{k=1}^n g_k g_{jk} ) \sum_{l=1}^n g_l g_{pl} \sum_{i=1}^n g_{mii}  -\frac{2}{I^{3/2}} g^3 \sum_{i,j=1}^n g_{ij} \sum_{k=1}^n g_k g_{pk} g_{mij} \\
&\ +\frac{1}{I^{3/2}} \sum_{j=1}^n \big( 2g_j g_{pj} +\sum_{k=1}^n (2g^3 g_k g_{jj} g_{pk} -g^3 g_j g_{jk} g_{pk} -g^3 g_k g_{jk}g_{pj} ) \big) \sum_{i=1}^n g_{mii} \\
&\ +\frac{3}{I^{5/2}} g^2 \sum_{i,j=1}^n \big( I g g_{ij} +g_i g_j + g^3 \sum_{k=1}^n (g_i g_j g_{kk} -g_i g_k g_{jk} -g_j g_k g_{ik} \big) \sum_{l=1}^n g_l g_{pl} g_{mij}  -\frac{1}{I^{3/2}} \sum_{i,j=1}^n (g_j g_{pi}+g_i g_{pj}) g_{mij}\\
&\ -\frac{1}{I^{3/2}} g^3 \sum_{k=1}^n g_{kk} \sum_{i,j=1}^n (g_j g_{pi}+g_i g_{pj})g_{mij} +\frac{1}{I^{3/2}} g^3 \sum_{i,j,k=1}^n\big((g_j g_{ik}+g_i g_{jk}) g_{pk} +g_k (g_{jk} g_{pi}+g_{ik} g_{pj})\big) g_{mij}  \\
&\ -\frac{3}{I^{5/2}} g\sum_{i,j=1}^n ( Ig g_{jj} +g_j^2 -g^3 g_j \sum_{k=1}^n g_k g_{jk} ) \sum_{l=1}^n g_l^2 g_p  g_{mii} +\frac{1}{I^{3/2}} \sum_{i,j=1}^n \big( Ig_{jj} +g^2 \sum_{k=1}^n ( 2g_k^2 g_{jj} -3 g_j g_k g_{jk}) \big) g_p  g_{mii} \\
&\ +\frac{3}{I^{5/2}} g \sum_{i,j=1}^n \big( I g g_{ij} +g_i g_j +g^3 \sum_{k=1}^n (g_i g_j g_{kk} - g_k g_j g_{ik} -g_k g_i g_{jk}) \big) \sum_{l=1}^n g_l^2 g_p g_{mij}  \\
&\ -\frac{1}{I^{3/2}} (2g^2 \sum_{k=1}^n g_k^2  +I) \sum_{i,j=1}^n g_{ij} g_p g_{mij}-\frac{3}{I^{3/2}} g^2 \sum_{i,j,k=1}^n g_i g_j g_{kk}g_p g_{mij} +\frac{3}{I^{3/2}} g^2 \sum_{i,j,k=1}^n g_k  (g_j g_{ik}+g_i g_{jk}) g_p g_{mij} \\
&\ +\frac{1}{2I^{5/2}} \sum_{i,j=1}^n g_{mi}\big( 4I(g_i g_{pjj}-g_j  g_{pij})  -2Ig^3 g_i \sum_{k=1}^n (g_{kk}g_{pjj}-g_{kj}g_{pkj}) \\
&\ \ -4Ig^2 g \sum_{k=1}^n g_k (g_{ik}g_{pjj}+g_{jj}g_{pik}-g_{jk}g_{pij}-g_{ij}g_{pjk}) -6g^2 g_i \sum_{k=1}^n ( g_k^2 g_{pjj} -g_j g_k g_{pjk}) \\
&\ \ +6g^5 g_i \sum_{k,l=1}^n g_k g_l ( g_{lk}g_{pjj} +g_{jj}g_{plk} -g_{jl}g_{pjk} -g_{jk}g_{pjl}) \big) \\
&\ +\frac{1}{I^{5/2}}g_m \sum_{i,j=1}^n \big(I (g_{jj}g_{pii}-g_{ij}g_{pij}) -3g^2g_i \sum_{k=1}^n  g_k (g_{ik}g_{pjj}-g_{jk}g_{pij} +g_{jj}g_{pik} -g_{ij}g_{pjk}) \\
&\ \ -3g \sum_{k=1}^n g_k^2 (g_i^2 g_{pjj} -g_j g_i g_{pij}) \big) , \\
\end{align*}
and the terms involving only first-order and second-order derivatives

\begin{align*}
&d_{mp} =\frac{1}{2I^{5/2}}\Big(I \sum_{i,j=1}^n (g_{ii}g_{jj}-g_{ij}^2)-6\big(g^2 \sum_{i,j,k=1}^n g_i g_k (g_{ik}g_{jj}-g_{jk}g_{ij}) +g  g_k^2 (g_i^2 g_{jj}- g_j g_i g_{ij}) \big) \Big)g_{mp}  \\
&\ +\frac{1}{2I^{5/2}} \sum_{i=1}^n g_{mi}\Big(4I\sum_{j=1}^n (g_{jj}g_{pi} -g_{ij}g_{pj}) -Ig^2 \sum_{j,k=1}^n \big( (g_{kk}g_{jj}-g_{kj}^2) gg_{pi} +4(g_{ik}g_{jj}-g_{jk}g_{ij})gg_{pk}\big) \\
&\ \ -2g^4 \sum_{j,k=1}^n \big( gg_i (g_{kk}g_{jj}-g_{kj}^2 ) +4gg_k (g_{ik}g_{jj}-g_{jk}g_{ij})\big) \sum_{l=1}^n g_l g_{pl} \\
&\ \ +6g^5 \sum_{j,k,l=1}^n ( g_{lk}g_{jj}- g_{jl}g_{jk}) (g_i g_l g_{pk}+g_i g_k g_{pl} +g_l g_k g_{pi} ) \\
&\ \ +g^2 \big(-6\sum_{j,k=1}^n ( g_k^2 g_{jj}-g_j g_k g_{jk} ) g_{pi} +8(g_i g_{jj} -g_j  g_{ij} )g_k g_{pk}-12g_i(g_j g_{kk}g_{pj} -g_k g_{jk} g_{pj} ) \big) \Big) \\
&\ -\frac{5}{2I^{7/2}} g^2 \sum_{i=1}^n g_{mi}\sum_{q=1}^n g_q g_{pq} \Big( 4I\sum_{j=1}^n (g_i g_{jj} -g_j  g_{ij} ) -Ig^2 \sum_{j,k=1}^n \big( gg_i (g_{kk}g_{jj}-g_{kj}^2) +4gg_k (g_{ik}g_{jj}-g_{jk}g_{ij})\big)  \\
&\ \ -6g^2g_i \sum_{j,k=1}^n (g_k^2 g_{jj} -g_j g_k g_{jk} ) +6g^5 g_i \sum_{j,k,l=1}^n ( g_l g_k g_{lk}g_{jj} -g_l g_k g_{jl}g_{jk}) \Big) \\
&\ +\frac{1}{2I^{5/2}}\sum_{i=1}^n g_{mi} g_p \sum_{j,k=1}^n \Big(8gg_k^2 (g_i g_{jj} -g_j  g_{ij} ) -2g\big( gg_i (g_{kk}g_{jj}-g_{kj}^2) +4gg_k (g_{ik}g_{jj}-g_{jk}g_{ij})\big)(g^2 \sum_{l=1}^n g_l^2+I) \\
&\ \ -Ig^2 g_i (g_{kk}g_{jj}-g_{kj}^2) -4Ig^2 g_k (g_{ik}g_{jj}-g_{jk}g_{ij}) -12g g_i (g_k^2 g_{jj} -g_j g_k g_{jk} ) \\
&\ \ +30g^4 g_i \sum_{l=1}^n (g_l g_k g_{lk}g_{jj} -g_l g_k g_{jl}g_{jk}) \Big)\\
&\ -\frac{5}{2I^{7/2}} g \sum_{q=1}^n g_q^2 \sum_{i=1}^n g_{mi} g_p \Big( 4I\sum_{j=1}^n (g_i g_{jj} -g_j  g_{ij} ) -Ig^2 \sum_{j,k=1}^n \big( gg_i (g_{kk}g_{jj}-g_{kj}^2) +4gg_k (g_{ik}g_{jj}-g_{jk}g_{ij})\big) \\
&\ \ -6g^2g_i  \sum_{j,k=1}^n (g_k^2 g_{jj} -g_j g_k g_{jk} ) +6g^5 g_i  \sum_{j,k,l=1}^n (g_l g_k g_{lk}g_{jj} -g_l g_k g_{jl}g_{jk}) \Big) \\
&\ +\frac{1}{I^{5/2}}g_m \sum_{i,j,k=1}^n \big( -3g^2 (g_{ik}g_{jj}-g_{jk}g_{ij})( g_i g_{pk} +g_k g_{pi}) +g^2 (g_{ii}g_{jj}-g_{ij}^2) g_k g_{pk}  \\
&\ \ -6g(g_i^2 g_{jj}- g_j g_i g_{ij}) g_k g_{pk} -3gg_k^2 (2g_i g_{jj}g_{pi}-g_j g_{ij} g_{pi} -g_i g_{ij}g_{pj} ) \big) \\
&\ -\frac{5}{2I^{7/2}} g^2 g_m \sum_{l=1}^n g_l g_{pl}\Big(I \sum_{i,j=1}^n (g_{ii}g_{jj}-g_{ij}^2) -6\big( g^2 \sum_{i,j,k=1}^ng_i g_k (g_{ik}g_{jj}-g_{jk}g_{ij}) +gg_k^2 (g_i^2 g_{jj}- g_j g_i g_{ij}) \big) \Big)   \\
&\ +\frac{1}{I^{5/2}}g_m g_p \sum_{i,j,k=1}^n \big( -6g g_i g_k(g_{ik}g_{jj}-g_{jk}g_{ij}) +g g_k^2(g_{ii}g_{jj}-g_{ij}^2) -3g_k^2 (g_i^2 g_{jj}- g_j g_i g_{ij})\big) \\
&\ -\frac{5}{2I^{7/2}} g \sum_{l=1}^n g_l^2 g_m g_p \Big(I \sum_{i,j=1}^n (g_{ii}g_{jj}-g_{ij}^2)-6\sum_{i,j,k=1}^n \big( g^2g_i g_k (g_{ik}g_{jj}-g_{jk}g_{ij}) +gg_k^2 (g_i^2 g_{jj}- g_j g_i g_{ij}) \big) \Big)  . \\
\end{align*}

\end{lemma}

\begin{proof}
Using Lemma \ref{lem:g_m} about the evolution of the first-order derivatives of \(g\), we have
\begin{equation}
(g_{mp})_t=L[g_{mp}] +c g_{mp} +\sum_{i,j=1}^n \partial_p(a_{ij})g_{mij} +\sum_{i=1}^n \partial_p(b_i)g_{mi}+\partial_p(c)g_m . \\
\end{equation} Then the equation (\ref{eq:g_mp}) is obtained by direct calculations.
\end{proof}

\subsection{Gradient estimates}
Now, we get estimates of the gradient of the pressure-like function \(g=\sqrt{2f}\) from above and below. To begin with, we assume that 
\begin{itemize}
\item the measure of \(\Sigma_0(T)\) is positive for \(0<T<T_c\) for some critical time \(T_c\), 
\item \(g\) is smooth up to the interface for a finite amount of time, 
\item \(g\) is of class \(C^{2+\beta}\) for some \(0<1< \beta\) for time \(0 \leq t \leq T < T_c\), 
\item \(g\) has first-order and second-order derivatives having the non-degeneracy condition at the initial time \(t=0\): there is a positive number \(\lambda\) such that the derivative \(\nabla g(x) \) and the second-order tangential derivatives \( \nabla^2_{\tau \tau} g(x) \) satistify \(
|\nabla g(x)| \geq \lambda \text{ and }  |\nabla^2_{\tau \tau} g(x)| \geq \lambda, \forall x \in \Gamma, \) 
\item \(f\) satisfies the non-degeneracy condition \(\max_{x\in \Omega(t)} f( x , t) \geq 2\) on \(0 \leq t \leq T < T_c \) so that g also meets the condition \( \max_{x\in \Omega(t)} g( x , t) \geq 2 \) in \( \Omega(t)=\{ x \in \mathbb{R}^n ; |\nabla f(x,t)|<\infty\} \). 
\end{itemize}
We approximate \(f\) by a decreasing sequence of positive, strictly convex and smooth functions \(f_\epsilon\) solving the equation (\ref{f_t}) of scalar curvature flow. We set \(g_\epsilon=\sqrt{2 f_\epsilon}\) and denote it simply by \(g\) afterwards. What we want to obtain is a gradient estimate from above.

\begin{lemma} \label{upperbound:Dg}
If the level set \(\Gamma_\epsilon(t)\) of \(g\) is convex, then there exists a constant \(C>0\) such that
\begin{equation}
|\nabla g|\leq C, \   \text{on} \ 0\leq g(\cdot,t) \leq 1, \ 0\leq t \leq T.\\
\end{equation}
\end{lemma}
\begin{proof}
Let us set \(X=\frac{1}{2}|\nabla g|^2=\frac{1}{2}\sum_i g_i^2\). Then \(X=\frac{1}{2}g^2_\nu \) for the normal vector \( \nu \) to the level set of \(g\). Suppose that at each time \(0\leq t <T\), \(X\) attains an interior maximum at \(P_0=(x_0, t) \) so that \(X(x_0,t)=\sup \{X(x,t);\ x\in \Sigma,\ 0\leq g(x,t)\leq 1\}\). Rotating the coordinates, we can make \(g_1=g_\nu>0\) and \(g_i=0\) for \(2\leq i \leq n\) at \(P_0\). As
\begin{equation} \begin{split}
X_1&=g_1 g_{11}+\sum_{i\geq 2} g_i g_{i1}=g_1 g_{11}=0, \  X_i=g_1 g_{1i}+\sum_{j\geq 2} g_j g_{ji}=g_1 g_{1i}=0, \ i \geq 2,
\end{split} \end{equation} we have \(g_{11}=g_{1i}=0\) for \(2\leq i \leq n\) at \(P_0\). So the nonzero second-order derivatives of \(g\) at \(P_0\) are \(g_{ij}\) for \(2\leq i,j \leq n\). In particular, \(g_{ii} \geq 0, \ 2\leq i \leq n\) by convexity of level sets of \(g\). \\

Now, let us look at the second-order derivatives of \(X\). At \(P_0\), we have
\begin{equation} \begin{split}
X_{11}&=g_1 g_{111}+g^2_{11}+\sum_{i\geq2}(g_i g_{i11}+g^2_{i1})=g_1 g_{111} \leq 0, \\
X_{ii}&=g_1 g_{1ii}+g^2_{1i}+\sum_{j\geq2}(g_j g_{jii}+g^2_{ji})=g_1 g_{1ii}+\sum_{j\geq2} g^2_{ji} \leq 0,  \ 2\leq i \leq n,  \\
X_{ij}&=g_1 g_{1ij}+g_{1i}g_{1j}+\sum_{k\geq2}(g_k g_{kij}+g_{ki}g_{kj})=g_1 g_{1ij}+\sum_{k\geq2} g_{ki}g_{kj} \leq 0,  \ 1\leq i,j \leq n,  \\
\end{split} \end{equation} so that \(g_1 g_{111}\leq 0\) and \(g_1 g_{1ii}\leq 0\). \\

Hence, the evolution of \(X\) at \(P_0\) is given by
\begin{equation} \begin{split}
X_t&=\bigg( \frac{g}{\sqrt{I}} \sum_{i=2}^n g_{ii} +\frac{g}{I^{3/2}} \sum_{i=2}^n g_{ii}+\frac{g_1^2}{I^{3/2}} \bigg) g_1 g_{111} \\ 
&\ \  +\frac{g_1^2}{I^{3/2}} \sum_{i=2}^n g_1 g_{1ii} +\frac{g}{\sqrt{I}} \sum_{i,j=2}^n g_1 g_{jj}g_{1ii} -\frac{g}{\sqrt{I}} \sum_{i,j=2}^n g_{ij} g_{1ij} \\ 
&\ \ -\frac{3g g_1^6}{ I^{5/2}} \sum_{i=2}^n g_{ii} +\frac{g_1^2}{2I^{3/2}} \sum_{i,j=2}^n (g_{ii} g_{jj} -g_{ij}^2)  \\
\end{split}\end{equation}
which is written as
\begin{equation}\begin{split}
X_t &=\bigg(\frac{g g_1 H_\epsilon}{I^{3/2}} +\frac{g g_1 H_\epsilon}{\sqrt{I}}+\frac{g^2_1}{I^{3/2}}\bigg) X_{11} + \sum_{i=2}^n \bigg(\frac{g g_1 H_\epsilon}{\sqrt{I}}+\frac{g^2_1}{I^{3/2}}\bigg) X_{ii} -\sum_{i,j=2}^n \frac{g g_{ij}}{\sqrt{I}} X_{ij} \\
&\ +\frac{-3g g_1^6}{I^{5/2}} g_1 H_\epsilon +\frac{1}{2 I^{3/2}} g_1^2(g_1^2 H_\epsilon^2-\sum_{i,j=2}^n g_{ij}^2)  \\
&\ - \sum_{i=2}^n \bigg(\frac{g g_1 H_\epsilon}{\sqrt{I}}+\frac{g^2_1}{I^{3/2}}\bigg)\sum_{j\geq2} g^2_{ji} +\sum_{i,j=2}^n \frac{g g_{ij}}{\sqrt{I}} \sum_{k\geq2} g_{ki}g_{kj}  
\end{split} \end{equation} where \(H_\epsilon=\frac{1}{g_1}\sum_{i=2}^n g_{ii} \) is the mean curvature of the level set \(\Gamma_\epsilon(t)\) at \(P_0\). \\

The level set \(\Gamma_\epsilon(t)\) of the convex function \(f\) is also convex. We can use local coordinates at \(P_0\) such that \(g_{ij}=0\) for \(i \neq j \) with \(g_1=g_\nu>0\) and \(g_i=0\) for \(2\leq i \leq n\). Then, since \(X_i=0\) and \(X_{ii} \leq 0\) for all \(1\leq i \leq n\) at \(P_0\),
\begin{equation} 
\begin{split}
X_t&=a_{11} X_{11}+ \sum_{i = 2}^n a_{ii} X_{ii} + \frac{g_1^2 }{2 I^{3/2}}(g_1^2 H_\epsilon^2-\sum_{i=2}^n g_{ii}^2) +\frac{-3g g_1^6}{I^{5/2}} g_1H_\epsilon -\sum_{i=2}^n \frac{g^2_1}{I^{3/2}} g_{ii}^2  \leq \frac{g_1^4}{2 I^{3/2}}H_\epsilon^2
\end{split}
\end{equation} 
where \(a_{11}=\frac{g g_1 H_\epsilon}{I^{3/2}} +\frac{g g_1 H_\epsilon}{\sqrt{I}}+\frac{g^2_1}{I^{3/2}} \geq 0\) and \(a_{ii}=\frac{g }{\sqrt{I}}\sum_{j \neq 1,i}  g_{jj}+\frac{g^2_1}{I^{3/2}} \geq 0\) for \(2 \leq i \leq n\). \\

In the previous sections, we proved that the speed of the level set is non-degenerate and finite:
\begin{equation}
\gamma(\theta, t_0) e^{-\frac{t-t_0}{B t_0}} \leq \gamma(\theta, t) \leq  \gamma(\theta, t_0) e^{-\frac{A(t-t_0)}{D+|B|T}}
\end{equation} which implies that
\begin{equation} \begin{split}
&\frac{\gamma(\theta, t_0) (e^{-\frac{t-t_0}{B t_0}}-1)}{t-t_0} \leq \frac{\gamma(t)-\gamma(t_0)}{t-t_0} \leq \frac{\gamma(t_0) (e^{-\frac{A(t-t_0)}{D+|B|T}}-1)}{t-t_0} , \\
&\gamma(t_0) \frac{-1}{|B|t_0} \leq \gamma'(t_0) \leq \gamma(t_0) \frac{-A}{D+|B|T}.
\end{split} \end{equation}

On the other hand, on the level set \(\Gamma_\epsilon(t)\),
\begin{equation}\begin{split}
&g_r \gamma'(t)+g_t=0, \ \ g_1=g_r D_{x_1} r=g_r \frac{x_1}{r} \Longrightarrow \ \ \frac{g_1}{g_t}= -\frac{x_1}{r}\frac{1}{\gamma'(t)}\ , \\ 
&0<C_1= \frac{|B|t_0}{R^2} \rho_0\leq \frac{|B|t}{\gamma^2 (x_0)} x_1 \leq \frac{g_1}{g_t} \leq \frac{D+|B|T}{A\gamma^2 (x_0)} x_1 \leq \frac{D+|B|T}{A\rho_0^2}R=C_2
\end{split}\end{equation} since the initial flat side \(\Sigma_0(0)\) is contained in the ball \(B_R\).  So we can say that \(\frac{g_1}{g_t}\) is of order 1 and we write \(\frac{g_1}{g_t} \sim 1\). \\

Since the level set \(\Gamma_\epsilon(t)\) is convex by assumption, it holds that at \(P_0\)
\begin{equation}
\sum_{i,j=2}^n (g_{ii} g_{jj}-g_{ii}^2)=2\sum_{2\leq i <j \leq n} g_{ii} g_{jj} \geq 0
\end{equation} and
\begin{equation} \label{ineq:bdryspeed}
C_1\leq \frac{g_t}{g_1}=\frac{1}{2 g_1 \sqrt{I}}g\sum_{i,j=2}^n (g_{ii} g_{jj}-g_{ij}^2)+\frac{1}{g_1 I^{3/2} } \sum_{i=2}^n \big( g(g_{ii} g_{11}-g_{i1}^2)+g_1^2 g_{ii}\big) \leq C_2.
\end{equation}

At \(P_0\), \(g_{11}=g_{i1}=0\) so we have
\begin{equation}
C_1 \leq \frac{g_t}{g_1} =\frac{g}{2g_1\sqrt{I}} \sum_{i,j=2}^n (g_{ii} g_{jj}-g_{ii}^2)+ \frac{g_1^2}{I^{3/2}} H_\epsilon \leq C_2,
\end{equation} and
\begin{equation} 
X_t \leq \frac{g_1^4}{2 I^{3/2}}H_\epsilon^2=\frac{I^{3/2}}{2} \bigg(\frac{g_1^2}{I^{3/2}} H_\epsilon\bigg)^2 \leq \frac{1}{2} C_2^2 I^{3/2}=\frac{1}{2} C_2^2 (1+g^2 g_1^2)^{3/2}.
\end{equation}

If \(g^2g_1^2 \leq 1\), then
\begin{equation}
X_t \leq \sqrt{2} C_2^2
\end{equation} and
\begin{equation}
X(t) \leq X(0)+t\sqrt{2} C_2^2 \leq X(0)+T\sqrt{2} C_2^2 .
\end{equation}

Otherwise, if \(g^2g_1^2 \geq 1\), then at \(P_0\) on \(\Gamma_\epsilon(t)\)
\begin{equation}
X_t \leq \sqrt{2}C_2^2 g^3 g_1^3=\sqrt{2}C_2^2 \epsilon^3 (2X)^{3/2}=4C_2^2 \epsilon^3 X^{3/2}.
\end{equation}

Hence, \(X(t)\) is bounded above by the solution \(Y(t)\) to the ODE \(Y_t=4C_2^2 \epsilon^3 Y^{3/2}\), \(Y(0)=X(0)\):
\begin{equation}
X(t) \leq  Y(t)=\frac{1}{(X(0)-2C_2^2 \epsilon^3 t)^2} \leq Y(T)=\frac{1}{(X(0)-2C_2^2 \epsilon^3 T)^2}
\end{equation} for time \(0\leq t \leq T\). \\

Taking \(\epsilon>0\) sufficiently small that \(2C_2^2 \epsilon^3 T \leq \frac{1}{2}X(0)\), we have
\begin{equation}
X(t) \leq \frac{4}{(X(0))^2}
\end{equation} for time \(0\leq t \leq T\). Hence, we have proved the lemma.
\end{proof}

\begin{lemma} \label{lowerbound:Dg}
If the level set \(\Gamma_\epsilon(t)\) of \(g\) is convex, then there exists a constant \(C>0\) such that
\begin{equation}
|\nabla g|\geq C, \ \  \text{on} \ 0\leq g(\cdot,t) \leq 1, \ 0\leq t \leq T.
\end{equation}
\end{lemma}

\begin{proof}
Let us define a quantity \(X=Y-kx \cdot \nabla f=x \cdot \nabla g-kx \cdot \nabla f\), where \(k\) is a positive constant that will be determined later. Suppose that at each time \(0\leq t <T\), \(X\) attains an interior minimum at \(P_0=(x_0, t)\) so that \(X(x_0,t)=\inf \{X(x,t);\ x\in \Sigma,\ 0\leq g(x,t)\leq 1\}\). We make \(x_1>0\) and \(x_i=0\) for \(i=2, \ldots , n\). \\

At \(P_0\) on \(\Gamma_\epsilon\),  we choose local coordinates such that the normal derivative is \(g_\nu=g_1\) and the tangential derivatives to the level set \(\Gamma_\epsilon\) vanish, in other words, \(g_i=0\) for \(i=2,\ldots,n\), and furthermore \(g_{ij}=0\) for all \(i\neq j\). Then we have \( \sum_{i,j=2}^n (g_{ii} g_{jj}-g_{ij}^2)=\sum_{i,j=2}^n g_{ii} g_{jj} - \sum_{i=1}^ng_{ii}^2=\sum_{i,j=2,\ i \neq j}^n g_{ii} g_{jj}\) and \(\sum_{i=2}^n (g_{ii}g_{11}-g_{1i}^2)=\sum_{i=2}^n g_{ii}g_{11}\) . \\

Then the inequality (\ref{ineq:bdryspeed}) becomes at \(P_0\),
\begin{equation}
C_1\leq \frac{g_t}{g_1}=\frac{1}{2 g_1 \sqrt{I}}g \sum_{i,j=2}^n g_{ii} g_{jj}+\frac{1}{g_1 I^{3/2} } \sum_{i=2}^n ( g g_{ii} g_{11}+g_1^2 g_{ii}) \leq C_2
\end{equation}
\begin{equation}
C_1\leq \frac{g_t}{g_1}=\frac{1}{2 g_1 \sqrt{I}}g\overline{R}_{g,2}+\frac{1}{I^{3/2} } g_1 \sum_{i=2}^n g_{ii} \leq C_2, 
\end{equation}
where
\begin{equation}
\overline{R}_{g,2}=\sum_{i,j=2}^n (g_{ii} g_{jj} -g_{ij}^2)+\frac{2}{I}\sum_{i=2}^n(g_{ii} g_{11} -g_{1i}^2). \\
\end{equation}

At \(P_0\),  \(x \cdot \nabla f=g x \cdot \nabla g=g x_1 g_1+g x_i g_i=g x_1 g_1\) in our local coordinates where \(g_i=0\) for \(i=2 \ldots n\). Also, we have \(X=x_1 g_1 -k g x_1 g_1\) and the first-order derivatives of \(X\) are 
\begin{equation}
\begin{split}
X_1&=g_1+x_1 g_{11}-k(g_1^2 x_1+g g_1+ g x_1 g_{11})=0, \\
X_i&=g_i+x_1 g_{1i}-k(g_i x_1 g_1+g g_i+ g x_1 g_{1i}) =x_1 g_{1i}-k g x_1 g_{1i}=0 \ \text{for} \ i=2, \ldots , n, \\
\end{split} 
\end{equation} and its second-order derivatives at \(P_0\) are 

\begin{equation} 
\begin{split} 
X_{11}&=x_1 g_{111} + 2g_{11}-k(3 x_1 g_1 g_{11}+2g_1^2+2g g_{11}+ g x_1 g_{111}),\\
X_{1i}&=x_1 g_{11i} + 2g_{1i}-k(2g_{1i} x_1 g_1+2g g_{1i}+ g x_1 g_{11i}),\  i=2, \ldots , n, \\
X_{ij}&=x_1 g_{1ij} + 2g_{ij}-k(g_{ij} x_1 g_1+g g_{ij}+ g x_1 g_{1ij}), \ i,j=2, \ldots , n. \\
\end{split}
\end{equation}

Then the evolution equation of \(X\) at the point \(P_0\) can be written as 
\begin{equation}
\begin{split}
X_t &=\bigg(\frac{g \sum_{j=2}^n g_{jj} }{\sqrt{I} } +\frac{g_1^2+g g_{11} }{ I^{3/2} }\bigg)\sum_{i =2}^n (X_{ii}-2g_{ii}+kg_{ii} x_1 g_1+kg g_{ii})+\frac{x_1 g_1}{2\sqrt{I} }  \overline{R}_{g,2}- \frac{kg x_1 g_1}{2\sqrt{I}} \overline{R}_{g,2} \\
&\ \  + \frac{g \sum_{i=2}^n g_{ii} }{ I^{3/2} }(X_{11}-2 g_{11}+3k x_1 g_1 g_{11} +2k g_1^2+2k g g_{11})\\
&\ \  +\frac{2g_1 \sum_{i=2}^n g_{ii} }{ I^{3/2} }(X_1-g_1+kg_1^2 x_1+kg g_1)- \left( \frac{1}{2\sqrt{I}} g\overline{R}_{g,2}+\frac{1}{I^{3/2}}g_1^2 \sum_{i=2}^n g_{ii} \right)k x_1 g_1 \\
&\ \  -\left(\frac{1}{2 I^{3/2}}g \overline{R}_{g,2} +\frac{1} { I^{5/2}}g_1^2 \sum_{i=2}^n g_{ii} \right) x_1 g g_1 (g_1^2+g g_{11}) \\
&\ \  +\left(\frac{1}{2 I^{3/2}}g \overline{R}_{g,2}+\frac{1} { I^{5/2}}g_1^2 \sum_{i=2}^n g_{ii} \right) k x_1 g^2 g_1(g_1^2+g g_{11})  \\
&\ \  -\frac{  2 x_1 g g_1 (g_1^2+g g_{11})^2 }{ I^{5/2} } \sum_{i=2}^n g_{ii} +\frac{  2k x_1 g^2 g_1 (g_1^2+g g_{11})^2 }{ I^{5/2}} \sum_{i=2}^n g_{ii} 
\end{split}
\end{equation}

which implies that
\begin{equation}
\begin{split}
X_t  &\geq k C_1 x_1 g_1^2+k C_1 g g_1 +\frac{k x_1 g_1}{ I^{3/2} }(g_1^2 +3g g_{11})\sum_{i=2}^n g_{ii}  +\frac{k g}{ I^{3/2} }(3g_1^2 +g g_{11}) \sum_{i=2}^n g_{ii} \\
&\ \  -3C_2 g_1 - C_2 k x_1 g_1^2 -\frac{3 C_2}{I} x_1 g g_1^2(g_1^2+g g_{11})+\frac{x_1 g_1}{2\sqrt{I} }  \overline{R}_{g,2} \\
&\ \  +\frac{C_1}{I} k x_1 g^2 g_1^2(g_1^2+g g_{11})  +\frac{  2k x_1 g^2 g_1 }{ I^{5/2} }  ( g_1^2+g g_{11})^2 \sum_{i=2}^n g_{ii} . \\
\end{split}
\end{equation}

On the other hand, we have at \(P_0\)
\begin{equation}
0=X_1=x_1(1-kg)g_{11}+g_1-k(g_1^2 x_1+g g_1)
\end{equation} so that by letting \( g \leq \frac{1}{2k}\) it holds that

\begin{equation}
g_{11}=\frac{-g_1+k(g_1^2 x_1+g g_1)}{x_1(1-kg)} \leq \frac{2k(C^2 R+C)}{\rho}
\end{equation} and that

\begin{equation}
\begin{split}
X_t &\geq  -k(C_2- C_1) C X -\frac{3C_2}{\rho}X - 3 C_2 C^3 X -3 C_2  C \frac{2k(C^2 R+C)}{\rho} X \\
&\ \  -\left( k(C_2- C_1) C+\frac{3C_2}{\rho}+(3 C_2-C_1)C^3 +(3 C_2-C_1)C \frac{2k(C^2 R+C)}{\rho} \right) k \rho C \\
&=-\alpha X -k\beta.
\end{split}
\end{equation}

Hence, by Grönwall's inequality, we obtain
\begin{equation}
\begin{split}
\text{min}_{x \in M} X(x,t) &\geq e^{-\alpha t}( \text{min}_{x \in M}X(x,0)-k \beta t ) \geq \frac{1}{2} e^{-\alpha T}\text{min}_{x \in M}X(x,0)=:C',
\end{split}
\end{equation} if we take \(k \leq \frac{\text{min}_{x \in M}X(x,0)}{2\beta T}\). Then \(Y=x_1 g_1 \geq  X \geq C'\) and \(|\nabla g|=g_1\geq \frac{C'}{R}\). We have obtained the lower bound of \(|\nabla g|\). \\
\end{proof}

\subsection{Second-order derivative estimates}

The purpose of this section is to obtain the optimal bound of the second-order tangential derivatives of \(g\) which gives the optimal decay rates of the second-order derivatives of \(f\). We get estimates on second-order derivatives of \(g\) first, under the assumptions that we have given in the subsection \ref{subsec:assump}. We also find a lower bound of \(R_{g,2}\), which is similar to the Aronson-Bénilan inequality \(\Delta u \geq -\frac{C}{t}\) for the porous medium equation.

\begin{lemma} Under the assumptions in the subsection \ref{subsec:assump}, there exists a constant \(c>0\) for which
\begin{equation}
c \leq g_t \leq c^{-1}.
\end{equation}
\end{lemma}

\begin{proof}
We have
\begin{equation}
g_r \dot{\gamma}(\theta, t)+g_t=0
\end{equation} so that
\begin{equation}
g_t=-g_r \dot{\gamma}(\theta, t).
\end{equation}
On the other hand, \(g_r\) is bounded, because
\begin{equation}
g_r=\nabla_x g \cdot \partial_r x=\frac{r g_i}{x_i} 
\end{equation} where \(g_i \) are bounded by Lemma \ref{upperbound:Dg} and \ref{lowerbound:Dg} and \(x_i\) are bounded by \( R>0\) in the subsection \ref{subsec:assump}.
Also, \(\dot{\gamma}(\theta, t)\) is bounded by two negative constants. From these bounds we get the result.
\end{proof}

\begin{corollary} Under the assumptions in the subsection \ref{subsec:assump}, there exists a constant \(C>0\) such that
\begin{equation} \label{ineq:f_ii f_jj -f_ij^2}
C \leq \frac{1}{2g\sqrt{1+|\nabla_x f|^2} } \sum_{i, j =1}^n (f_{ii} f_{jj}-f_{ij}^2) \leq C^{-1}.
\end{equation}
\end{corollary}

\begin{proof}
Recall that, in the coordinates where \(g_i=g_1 \delta_{1i}\) so that \(f_i=f_1 \delta_{1i}\),
\begin{equation} 
\begin{split}
f_t &= \frac{1}{\sqrt{1+|\nabla_x f|^2}}(1-\frac{f_1^2}{1+|\nabla_x f|^2})\sum_{ i=2}^n (f_{ii} f_{11}-f_{i1}^2) +\frac{1}{2\sqrt{1+|\nabla_x f|^2} } \sum_{i, j =2}^n (f_{ii} f_{jj}-f_{ij}^2) \\
\end{split} 
\end{equation}
we have
\begin{equation}
\frac{1}{2 (1+|\nabla_x f|^2)^{3/2} } \sum_{i, j =1}^n (f_{ii} f_{jj}-f_{ij}^2)\leq f_t \leq \frac{1}{2\sqrt{1+|\nabla_x f|^2} } \sum_{i, j =1}^n (f_{ii} f_{jj}-f_{ij}^2).
\end{equation}

Since \(f_t=g g_t\),  we get the desired result as \(g\) tends to \(0\) near the interface.
\end{proof}

\begin{lemma}
Let us assume the conditions in the subsection \ref{subsec:assump}. Then there exists a constant \(C>0\) satisfying
\begin{equation*}
0<\sum_{i \neq 1} g_{ii} \leq C.
\end{equation*}
for sufficiently small value of \(g\) near the interface.
\end{lemma}

\begin{proof}
Let us define a quantity 
\begin{equation}\begin{split}
X&:= \sum_{i,j=1}^n (g_i^2 g_{jj}-g_i g_j g_{ij})+\Delta f = \sum_{i,j=1}^n (g_i^2 g_{jj}-g_i g_j g_{ij})+\sum_{i=1}^n (g_i^2+g g_{ii}) .\\
\end{split} \end{equation} 

We work at an interior point \(P_0\) of the domain, where \(X\) takes its supremum. At \(P_0\), we take the coordinate system where \(g_i=0\) for \(i \neq 1\) and \(e_1=\nu\), the normal direction to \(\Gamma_\epsilon\), and \(g_{ij}=0\) for \(i \neq j\).  We assume that \(P_0\) is on a level set \(\Gamma_\epsilon\). Then we have 
\begin{equation}
X= g_1^2 \sum_{i =2}^n g_{ii}+(g_1^2 +g g_{11} +g \sum_{i=2}^n g_{ii}) .\\
\end{equation}

In Lemma \ref{upperbound:Dg} and \ref{lowerbound:Dg}, we have shown \(|\nabla g|= g_1\) is bounded from above and below; there exist a constant \(c>0\) such that \(0<c\leq g_1 \leq c^{-1}\) on \(g>0, \ 0\leq t \leq T\). Also, \(\Delta f=g_1^2+g \sum_{i=1}^n g_{ii}=g_1^2+g g_{11}+g \sum_{i=2}^n g_{ii} \) is bounded as \(f \in C^{1,1}\). Hence, an upper bound of \(X\) will give an upper bound of the tangential Laplacian \(\Delta_\tau g =\sum_{i=2}^n g_{ii} \). \\

Differentiating \(X\) with respect to space variables at \(P_0\), we get for all \(1\leq k \leq n\)
\begin{equation}
0=X_k =\sum_{i,j=1}^n (g_j^2g_{iik}-g_i g_j g_{jik}+2g_j g_{ii} g_{jk} -2g_j  g_{ji} g_{ik}) +\sum_{i=1}^n (g g_{iik}+2g_i g_{ik}+g_k g_{ii}) .
\end{equation}

We have, at \(P_0\),
\begin{equation}
\begin{split}
0=X_1&=g g_{111}+(g_1^2+g)\sum_{i \neq 1} g_{ii1} +2g_1 g_{11}\sum_{i \neq 1}g_{ii} +3g_1 g_{11}+g_1 \sum_{i \neq 1} g_{ii} ,\\
0=X_k&=g_1^2 \sum_{i=1}^n g_{iik}-g_1^2 g_{11k} +\sum_{i=1}^n (g g_{iik})=g g_{11k} +(g_1^2+g)\sum_{i \neq 1}g_{iik}  ,\\
\end{split}
\end{equation}
and then
\begin{equation}
\begin{split}
\sum_{i \neq 1} g_{ii1}&=-\frac{1}{g+g_1^2} ( g g_{111} +2g_1 g_{11}\sum_{i \neq 1} g_{ii} +3g_1 g_{11}+\sum_{i \neq 1} g_1 g_{ii} ) ,\\
\sum_{i \neq 1} g_{iik}&=-\frac{1}{g+g_1^2} g g_{11k} , \text{ for } k \neq 1 . \\
\end{split}
\end{equation}

Differentiating  \(X\) with respect to time, we get \(X_t\) as follows. We use the indices \(m, p\) instead of \(i,  j\) in order to avoid any possible confusions. The right side of the evolution equation of \(X\) will contain spatial derivatives of \(g\) of order no greater than 4, because the equation (\ref{eq:g_mp}) is a second-order equation. Using an index \(2 \leq m \leq n\), we can simply write the evolution of \(X\) as the following:
\begin{equation}
\begin{split}
X_t&=g_1^2 \sum_{m \neq 1}( g_{mm})_t+2g_1 (g_1)_t \sum_{m \neq 1} g_{mm} +\big(g (g_{11})_t+g \sum_{m \neq 1} (g_{mm})_t+2g_1 (g_1)_t+g_t g_{11}+g_t \sum_{m \neq 1} g_{mm} \big) . \\
\end{split}
\end{equation} 

Now, let us consider the second-order derivatives of \(X\). They are
\begin{equation}
\begin{split}
&X_{kl} =\sum_{i,j=1}^n (g_j^2g_{iikl}-g_i g_j g_{ijkl} ) \\
&\   +\sum_{i,j=1}^n (2g_j g_{jl} g_{iik}-2g_j g_{il} g_{ijk} +2g_j g_{jk} g_{iil}-2g_j g_{ik} g_{ijl} +2g_j g_{ii}g_{jkl} -2g_j g_{ji}g_{ikl}  +2g_{jk} g_{jl}g_{ii}-2g_{jk} g_{il} g_{ij}) \\
&\  +\sum_{i=1}^n (g g_{iikl}+g_k g_{iil}+g_l g_{iik}+2g_i g_{ikl}+g_{kl} g_{ii}+2g_{ik} g_{il}) .\\
\end{split}
\end{equation}

At \(P_0\) under the coordinates \(g_k=0\) for \(k \neq 1\) and \(g_{ij}=0\) for \(i \neq j\),
\begin{equation}
\begin{split}
X_{kl}&=g_1^2 \sum_{i =2}^n g_{iikl} +2g_1 g_{1l} \sum_{i=1}^n g_{iik}+2g_1 g_{1k} \sum_{i=1}^n g_{iil}  +2g_1 \sum_{i \neq 1,k,l}g_{ii} g_{1kl}  +2g_{kl} g_{kk} \sum_{i \neq k} g_{ii} \\
&\ \  +(g \sum_{i=1}^n g_{iikl}+2g_1 g_{1kl} +g_k \sum_{i=1}^ng_{iil} +g_l \sum_{i=1}^n g_{iik}+g_{kl} \sum_{i=1}^n g_{ii}+2g_{kl} g_{kk} ) \\
\end{split}
\end{equation}

In particular,
\begin{equation}
\begin{split}
X_{11}&=g_1^2 \sum_{i=2}^n g_{ii11} +4g_1 g_{11} \sum_{i=1}^n g_{ii1} +2g_1 (\sum_{i=2}^n g_{ii} -2g_{11}  ) g_{111} +2g_{11}^2 \sum_{i=2}^n g_{ii} \\
&\ \  +(g \sum_{i=1}^n g_{ii11}+2g_1 g_{111}+2g_1 \sum_{i=1}^n g_{ii1}+g_{11} \sum_{i=1}^n g_{ii}+2g_{11}^2 ) ,\\
\end{split}
\end{equation}

and for \(k \neq 1\) 
\begin{equation}
\begin{split}
X_{kk}&=g_1^2 \sum_{i=2}^n g_{iikk} +2g_1 (\sum_{i\neq 1,k} g_{ii} -g_{kk}) g_{1kk} +2g_{kk}^2 \sum_{i\neq k} g_{ii} \\
&\ \  +(g \sum_{i=1}^n g_{iikk}+2g_1 g_{1kk}+g_{kk} \sum_{i=1}^n g_{ii}+2g_{kk}^2 ) .\\
\end{split}
\end{equation}

For \(k \neq 1\)
\begin{equation}
\begin{split}
X_{1k}&=g_1^2 \sum_{i=2}^n g_{ii1k}+2g_1 g_{11}\sum_{i=1}^n g_{iik} +2g_1 (\sum_{i=1,k}^n g_{ii} -g_{11} ) g_{11k}  \\
&\ \ +(g \sum_{i=1}^n g_{ii1k}+2g_1 g_{11k}+g_1 \sum_{i=1}^n g_{iik} ) ,\\
\end{split}
\end{equation}

and for \(2 \leq k,l \leq n\) such that \(k \neq l \)
\begin{equation}
\begin{split}
X_{kl}&=g_1^2 \sum_{i=2}^n g_{iikl}+2g_1 \sum_{i\neq1,k,l} g_{ii} g_{1kl} +(g \sum_{i=1}^n g_{iikl}+2g_1 g_{1kl}) .\\
\end{split}
\end{equation}

In our coordinate system at \(P_0\), the quantity \(J=|\nabla g|^2+g\) becomes \(J=g_1^2+g\). Then the term \(\frac{J}{g}\) goes to \(+\infty\) as \(g\) tends to \(0^+\) and \(\Delta f =g_1^2 +g g_{11} + g \sum_{i \neq 1 } g_{ii}\) is bounded so that \(0\leq f_{11}=g g_{11}+g_1^2 \leq C\) and \(0\leq f_{ii}=g g_{ii}\leq C\) for \(i \neq  1\). We have the relations \(0=X_m= \sum_{i \neq 1} (g_1^2+g)g_{iim}+g g_{11m}\) for \(m \neq 1\) and \(0=X_1=(g_1^2 +g) \sum_{ i\neq 1} g_{ii1}+g g_{111}+g_1 (1+2g_{11})\sum_{ i \neq 1} g_{ii}+3g_1 g_{11}\), so that
\begin{equation}
\begin{split}
 J\sum_{i \neq 1} g_{ii1}&=-\big(g g_{111} +(g_1+2g_1 g_{11})\sum_{i \neq 1} g_{ii}+3g_1 g_{11} \big) , \\
J \sum_{i \neq 1}g_{iik}&=-g g_{11k} , \text{ for } k \neq 1 .\\
\end{split}
\end{equation}

As a result, we also have relations
\begin{equation}
\begin{split}
g_1 g_{11} \sum_{i \neq 1} g_{ii} g g_{111}&=-Jg_1 g_{11} \sum_{j \neq 1} g_{jj} \sum_{i \neq 1} g_{ii1} -g_1 g_{11} \sum_{j \neq 1} g_{jj}\big((g_1+2g_1 g_{11})\sum_{i \neq 1} g_{ii}+3g_1 g_{11} \big) , \\
g_1 g g_{111}&=-J g_1 \sum_{i \neq 1} g_{ii1} -g_1 \big((g_1+2g_1 g_{11})\sum_{i \neq 1} g_{ii}+3g_1 g_{11} \big) . \\
\end{split}
\end{equation}

We use the finite and non-degenerate speed of the level set of \(g\) to bound the term \(\Big( \sum_{i \neq 1} \frac{g_1^2} {g}  g_{ii} \Big)^2\) above. Previously, we obtained the result that for some sufficiently small constant \(0<C<1\),
\begin{equation}\begin{split}
C&\leq \frac{g_t}{g_1}=\frac{1}{2 g_1 \sqrt{I}}g \sum_{ 2 \leq i ,j } (g_{ii} g_{jj}-g_{ij}^2)+\frac{1}{g_1 I^{3/2} } ( g g_{11} \sum_{i \neq 1} g_{ii}+g_1^2 \sum_{i \neq 1} g_{ii}) \leq C^{-1}.
\end{split}\end{equation}
and we also see that, by the boundedness of \(g_1\), 
\begin{equation}\begin{split}
C&\leq \frac{1}{2 \sqrt{I}}g \sum_{2 \leq i , j } (g_{ii} g_{jj}-g_{ij}^2)+\frac{1}{I^{3/2} } ( g g_{11} \sum_{i \neq 1} g_{ii}+g_1^2 \sum_{i \neq 1} g_{ii}) \leq C^{-1}.
\end{split}\end{equation}

This also implies that
\begin{equation}\begin{split}
\frac{1}{2 \sqrt{I} } g R_{g,2} +\frac{1}{I^{3/2} } g_1^2 \sum_{i \neq 1} g_{ii} \leq C^{-1},
\end{split}\end{equation}

\begin{equation}\begin{split}
\frac{1}{2 I^{3/2} } g R_{g,2} +\frac{1}{I^{3/2} } g_1^2 \sum_{i \neq 1} g_{ii}  \leq C^{-1},
\end{split}\end{equation}
and
\begin{equation}\begin{split}
0&\leq \frac{1}{I^{3/2} } ( g g_{11}+g_1^2) \sum_{i \neq 1} g_{ii}  \leq C^{-1}.
\end{split}\end{equation}

Since \(\Delta f\) is bounded and \(X =g_1^2 \sum_{i \neq 1} g_{ii} +\Delta f \) in our coordinates at \(P_0\), without loss of generality we may assume that 
\begin{equation}
C^{-1} \leq \frac{1}{2I^{3/2} } g_1^2 \sum_{i \neq 1} g_{ii} \leq X \leq 2g_1^2 \sum_{i \neq 1} g_{ii},
\end{equation}
because, if \(\frac{1}{2I^{3/2} } g_1^2 \sum_{i \neq 1} g_{ii}\leq  C^{-1} \), then \(X\) is bounded above and we get the upper bound of \(\sum_{i \neq 1} g_{ii}\), which is our goal. \\

So we have
\begin{equation}
\begin{split}
0\leq \frac{1}{I^{3/2} } g_1^2 \sum_{i \neq 1} g_{ii} \leq C^{-1}-\frac{1}{2 \sqrt{I}}g \sum_{2 \leq i , j } (g_{ii} g_{jj}-g_{ij}^2)-\frac{1}{I^{3/2} } g g_{11} \sum_{i \neq 1} g_{ii}
\end{split}
\end{equation}
and
\begin{equation}
\begin{split}
\frac{1}{2 \sqrt{I}}g \sum_{ 2 \leq i , j } (g_{ii} g_{jj}-g_{ij}^2)+\frac{1}{I^{3/2} } g g_{11} \sum_{i \neq 1} g_{ii} \leq C^{-1} -\frac{1}{I^{3/2} } g_1^2 \sum_{i \neq 1} g_{ii} \leq -\frac{1}{2I^{3/2} } g_1^2 \sum_{i \neq 1} g_{ii} .
\end{split}
\end{equation}

Hence, we get
\begin{equation}
\begin{split}
0 \leq \frac{1}{2I^{3/2} } g_1^2 \sum_{i \neq 1} g_{ii} \leq -\frac{1}{2 \sqrt{I}}g \sum_{ 2 \leq i , j } (g_{ii} g_{jj}-g_{ij}^2)-\frac{1}{I^{3/2} } g g_{11} \sum_{i \neq 1} g_{ii} ,
\end{split}
\end{equation}

\begin{equation}
\begin{split}
0 \leq  \frac{g_1^2}{g} \sum_{i \neq 1} g_{ii} \leq -I \sum_{ 2 \leq i , j } (g_{ii} g_{jj}-g_{ij}^2) -2g_{11} \sum_{i \neq 1} g_{ii} =-IR_{g,2}+2g^2 g_1^2 g_{11} \sum_{i \neq 1} g_{ii}.
\end{split}
\end{equation}
In particular, since the level set of \(g\) is convex, we get
\begin{equation}
\begin{split}
2 g_{11} \sum_{i \neq 1} g_{ii} \leq  -I \sum_{ 2 \leq i , j } (g_{ii} g_{jj}-g_{ij}^2) \leq 0,
\end{split}
\end{equation}

\begin{equation}
\begin{split}
IR_{g,2} \leq 2g^2 g_1^2 g_{11} \sum_{i \neq 1} g_{ii} \leq 0
\end{split}
\end{equation}
so that \(g_{11} \leq 0\) and \(R_{g,2} \leq 0 \). 

As \(g\) tends to zero, \(LX\) at \(P_0\) becomes
\begin{align*}
&LX:=X_t-a_{ij} D_{ij} X = -\big( 31J -20g I \big) \frac{1}{20 J^2 I^{3/2}} g^3 g_{111}^2 -\frac{1}{J I^{3/2}} \big( 2g J^2 +g^2 (4J -I g ) \big) \sum_{i \neq 1} g_{i11}^2 \\ 
&\ -\frac{1}{J \sqrt{I}} g \sum_{i,j,k \neq 1} (J g_{ijk})^2 -\frac{1}{J^2 \sqrt{I}} g^2 \sum_{i,j \neq 1} (J g_{ij1})^2 -\frac{2}{I^{3/2}} Jg \sum_{i,j \neq 1, i \neq j} g_{ij1}^2 -\frac{1}{4I^{3/2}} Jg \sum_{i \neq 1} g_{ii1}^2  \\ 
&\ -\frac{1}{20J I^{3/2}} g \Big( \big( g g_{111} -\frac{20J}{I} g_1^3 g \sum_{j \neq 1} g_{jj} \big)^2 +(g g_{111} +5Jg_1 g^2 R_{g, 2})^2  +(g g_{111} -10Jg_1 g g_{11} g \sum_{j \neq 1} g_{jj} )^2  \\
&\ \ +\big(g g_{111} +\frac{90J}{I} g_1 (g g_{11}+g_1^2) g \sum_{j \neq 1} g_{jj} \big)^2 +\big(g g_{111} -\frac{20I}{J} g_1 g \sum_{ i \neq 1} g_{ii} \big)^2 +\big(g g_{111} -\frac{60I}{J} g_1 g g_{11} \big)^2  \\
&\ \ +(g g_{111} +20g_1 \sum_{ i \neq 1} g_{ii} )^2 +(g g_{111} +60g_1 g_{11} )^2  +\big(g g_{111}-\frac{40I}{J} g_1 g g_{11}\sum_{ i \neq 1} g_{ii} \big)^2 \Big)  \\ %
&\ -\frac{1}{8I^{3/2}} Jg \sum_{i \neq 1} \Big( \big( g_{ii1} +\frac{8}{J} g g_1 g^2 g_{11} \sum_{j \neq 1} g_{jj} \big)^2 +\big(g_{ii1} -\frac{24}{J} g_1 g_{11}  \big)^2 +\big(g_{ii1} +\frac{24}{JI} g g_1 f_{11}^2 \big)^2 \\
&\ \ +( g_{ii1} +2g_1 g^2 R_{g,2} )^2 +\big(g_{ii1} +\frac{4}{I} g_1^3 g \sum_{j\neq 1} g_{jj} \big)^2  +\big(g_{ii1} -\frac{4}{I} g^2 g_1^3 g^2 g_{11} \sum_{j \neq 1} g_{jj} \big)^2  +\big(g_{ii1} +\frac{16}{J} g_1 g_{ii}^2 \big)^2  \\
&\ \ +\big(g_{ii1} +\frac{8}{I} g_1 g^2 g_{11} \sum_{j \neq 1} g_{jj} \big)^2 +\big(g_{ii1} -\frac{16I}{J} g_1 \sum_{j \neq 1} g_{jj} g_{ii} \big)^2  +\big( g_{ii1} -\frac{16}{J} g_1 g_{11} g_{ii} \big)^2  \\
&\ \ +\big( g_{ii1} +\frac{24}{J} g_1 g_{ii} \big)^2  +\big(g_{ii1} -\frac{8}{J} g g_1 g^2 g_{11}g_{ii}\big)^2  +(g_{ii1} +16g_1 g^2 \sum_{j \neq 1} g_{jj} g_{ii})^2 +\big(g_{ii1} -\frac{16}{J} g g_1 g^2 g_{ii}^2 \big)^2 \Big) \\ 
&\ +\Bigg( \frac{1}{I^{5/2}} \mathcal{O}(1) +\frac{n-1}{8I^{3/2}} J\mathcal{O}(g) +\frac{1}{8I^{3/2}}\mathcal{O}(g) +\frac{5}{2I^{7/2}} \mathcal{O}(g^2)\Bigg) +\frac{1}{2I^{3/2}} \big( ( J +2g_1^2 ) \sum_{i \neq 1} g_{ii} +\mathcal{O}(1) \big) R_{g,2} \\ 
&\ +\Bigg(  \frac{4}{J^2\sqrt{I}}  \mathcal{O}(1)  g_{11} +\frac{1}{JI^{3/2} } \Big( g_1^2 (g g_{11}) \big(180+ 72(n-1)\big) -2J^2 ( g^3 \sum_{i \neq 1} g_{ii}^3 ) \Big) g_{11} -\frac{3}{I^{3/2} } (g \sum_{j \neq 1} g_{jj}) g_{11}^2 \Bigg)  \\  
&\ + \Bigg( \frac{2}{I^{3/2}} \mathcal{O}(1) \sum_{i \neq 1} g_{ii} +\frac{1}{2I^{3/2}}\mathcal{O}(1) ( g \sum_{i \neq 1} g_{ii}^2 ) +\frac{1}{I^{3/2}}\mathcal{O}(1) (g^3 \sum_{i \neq 1} g_{ii}^4) \\ 
&\ -\frac{1}{I^{3/2}} g_1^2 (4+I) (-g_{11}) \sum_{i \neq 1} g_{ii} -\frac{4}{JI^{3/2} } g_1^2 (-g g_{11}) (\sum_{i \neq 1} g_{ii})^2 -\frac{1}{\sqrt{I} } (-g g_{11}) \sum_{i \neq 1} g_{ii}^2 \\  
&\ -\frac{1}{I^{3/2} } f_{11} \big( 2\sum_{i \neq 1} g_{ii}^2+ (\sum_{i \neq 1} g_{ii} )^2 \big) -\frac{2}{I^{5/2}} g_1^2 (\sum_{i \neq 1} g_{ii})^2 -\frac{1}{\sqrt{I} } (g \sum_{j \neq 1} g_{jj}) \big( (\sum_{i \neq 1} g_{ii})^2 +2\sum_{i \neq 1} g_{ii}^2 \big) \Bigg) \\ 
&\ +\Bigg( \bigg( \frac{1}{I^{3/2} } \mathcal{O}(1) ( \sum_{ i \neq 1} g_{ii} )^2 +\frac{2}{\sqrt{I} } g \sum_{i \neq 1} g_{ii}^3 +\frac{1}{\sqrt{I} } \mathcal{O}(1) \sum_{i \neq 1}  g_{ii}^2  \\ 
&\ +\frac{2}{ I^{3/2} } (Ig g_{11}-J) \sum_{i \neq 1} g_{ii}^3 -\frac{2}{\sqrt{I} } g \sum_{i \neq 1}  g_{ii}^4  -\frac{2}{I^{3/2} } (I g \sum_{k \neq 1} g_{kk}+f_{11} )\sum_{i \neq 1} g_{ii}^2 \sum_{j \neq 1} g_{jj}  \bigg) \\
&\ +\bigg( \frac{2}{I^{3/2}} \mathcal{O}(1) g_{11} \sum_{j\neq 1}g_{jj} +\frac{2}{I^{3/2}} Jg_{11} \sum_{i \neq 1} g_{ii}^2 +\frac{4}{I^{3/2}} g_1^2 g_{11} (\sum_{i \neq 1} g_{ii})^2 \bigg)  \Bigg) \\ 
&\ +\Bigg( \bigg( \frac{1}{JI^{3/2} } \mathcal{O}(1) g_{11} \sum_{i \neq 1} g_{ii}^2 +\frac{1}{I^{5/2}} \mathcal{O}(1) g_{11} ( \sum_{i \neq 1} g_{ii})^2 \bigg)-\frac{1}{I^{3/2} } g g_{11}^2 (\sum_{j \neq 1} g_{ii})^3 \\ 
&\ +\bigg( \frac{1}{JI^{3/2} } \mathcal{O}(1)\sum_{i \neq 1} g_{ii} \sum_{j \neq 1} g_{jj}^2 +\frac{32}{JI^{3/2} } g_1^2 g \sum_{i \neq 1} g_{ii}^4 +\frac{4}{\sqrt{I} } \mathcal{O}(1) \sum_{i \neq 1}  g_{ii}^3 \bigg)-\frac{1}{I^{3/2} } g g_{11}^2 (\sum_{j \neq 1} g_{ii})^3 \Bigg) \leq C \\ 
\end{align*}
for some uniform constant \(C>0\), hence \(X\) satisties
\begin{equation}
X_t\leq LX \leq C
\end{equation}
at its maximum point \(P_0\) at each time, so \(X\) is bounded above. Consequently, \(\sum_{i \neq 1} g_{ii}\) is also bounded above.

\end{proof}

From the inequality (\ref{ineq:f_ii f_jj -f_ij^2}), it holds that for some constant \(C>0\)
\begin{equation}
\sum_{i,j}(f_{ii} f_{jj}-f_{ij}^2)=2\sum_{i \neq 1}(f_{11}f_{ii}-f_{i1}^2)+\sum_{i,j \neq 1}(f_{ii} f_{jj}-f_{ij}^2)  \geq Cg
\end{equation}
which yields
\begin{equation} \label{ineq:f_ii^2+2f_11 f_ii}
\big(\sum_{i \neq 1}f_{ii} \big)^2 +2f_{11}\sum_{i \neq 1}f_{ii} =(\sum_{i \neq 1}f_{ii} +2f_{11}) g\sum_{i \neq 1} g_{ii} \geq Cg.
\end{equation}

As \(\Delta f\) is bounded, there exists some constant \(\widetilde{C}>0\) such that
\begin{equation}
\sum_{i \neq 1}  g_{ii} \geq \widetilde{C},
\end{equation}
which implies the following corollary.
\begin{corollary} \label{bound:g_ii}
Let us assume the conditions in the subsection \ref{subsec:assump}. Then there exists a constant \(C>0\) satisfying
\begin{equation} \label{inequal:g_ii}
C\leq \sum_{i \neq 1} g_{ii} \leq C^{-1}.
\end{equation}
for sufficiently small value of \(g\) near the interface.
\end{corollary}

We can also examine the bounds for the second-order derivatives of \(f\). 
\begin{corollary} \label{bound:D^2 f}
Let us assume the conditions in the subsection \ref{subsec:assump}. Then there exists a constant \(C>0\) satisfying
\begin{equation}
C\leq f_{11} \ , \frac{1}{g}{\sum_{i \neq 1} f_{ii}} \leq C^{-1} \ \text{    and    } \ \sqrt{ \frac{1}{ g }\sum_{i \neq 1}f_{i1}^2}\leq C^{-1}.
\end{equation}
for sufficiently small value of \(g\) near the interface.
\end{corollary}

\begin{proof}
From the inequality (\ref{inequal:g_ii}), we see that
\begin{equation} \label{inequal:f_ii}
C \leq \frac{1}{g}{\sum_{i \neq 1} f_{ii}} \leq C^{-1}.
\end{equation}

With the convexity of \(f\) and the boundedness \(0\leq f_{11} \leq C^{-1} \), and the inequality (\ref{ineq:f_ii^2+2f_11 f_ii}) implies
\begin{equation}
\sum_{i \neq 1}f_{i1}^2 \leq \sum_{i \neq 1} f_{11}f_{ii} \leq C^{-1} g.
\end{equation}

Thus
\begin{equation}
\sqrt{ \frac{1}{ g }\sum_{i \neq 1}f_{i1}^2}\leq C^{-1}.
\end{equation}

Since \(Cg^2 \leq \big( {\sum_{i \neq 1} f_{ii}} \big)^2 \leq C^{-1}g^2\) because of the inequality (\ref{inequal:f_ii}) as \(g\) tends to \(0\) near the interface, we see that
\( f_{11} \geq C\) from the inequality (\ref{ineq:f_ii^2+2f_11 f_ii}).
\end{proof}

Let \(P=(x,t) \) where \( x \in \Omega(t)=\{ x \in \mathbb{R}^n ; |\nabla f| < \infty \}\). Then we can define the set \(\Omega_P (t)=\{ x \in \Omega(t) ; f(x,t) \leq f(P) \}\) and the level curve \( \Gamma_P (t)=\partial \Omega_P (t) \) for \(P\). Note that the interface is \(\Gamma(t)=\partial \Omega(t)\). Let \(\nu(P)\) be the normal vector to  \( \Gamma_P (t)\) at \(P\) and \(\bf{P}\) be the position vector of \(P\) with respect to the origin \(O\). Then, using the condition that the disk \(B(0,\rho_0)\) is always contained in the interior of the interface \(\Gamma(t)\), we get the following result.

\begin{lemma}  \label{bound:support'}
Let us assume the conditions in the subsection \ref{subsec:assump}. Then we have
\begin{equation}
\textbf{P} \cdot \nu (P) \geq \rho_0, \ \ \ \text{for any } P=(x,t) \text{ where } x \in \Omega(t), \ 0 \leq t \leq T .
\end{equation}
\end{lemma}

Let \(P_0=(x,t_0)\) where \( x\) an arbitrary point on the interface \(\Gamma(t_0)\) with \(0 < t_0 \leq T\). We define its position vector \(\bf{P_0}\) and normalize it with \(n_0 =\frac{\bf{P_0}}{|\bf{P_0}|}\). Then we obtain the following.

\begin{lemma} \label{bound:support}
Let us assume the conditions in the subsection \ref{subsec:assump}. Then there exists positive constants \(\eta\) and \(\gamma\) such that
\begin{equation}
n_0 \cdot \nu (P)\geq \gamma, \ \ \ \text{for any } P=(x,t),  \text{ where } x \in \Omega(t),\ |P-P_0|\leq \eta,\  0\leq t \leq t_0.
\end{equation}
\end{lemma}

\begin{proof}
Note that \(|\textbf{P}|\leq C\) for some constant \(C>0\) which depends on initial data. With Lemma \ref{bound:support'}, it follows that
\begin{equation}
\frac{1}{|\textbf{P}|} \textbf{P} \cdot \nu (P) \geq C^{-1} \rho_0.
\end{equation}
Setting \(\gamma=C^{-1} \rho_0/2\) and choosing \(\eta\) sufficiently small depending on \(\rho_0\) and \(\gamma\), we get the desired result.
\end{proof}

Consequently, the following result holds for the derivatives of \(g\) and \(f\) in the direction \(n_0=\frac{\bf{P_0}}{|\bf{P_0}|}\) for any point \(P_0 =(x_0 ,t) \) where \( x_0 \in \Gamma(t)\) on the interface.
\begin{lemma}
There exist positive constants \(C\) and \(\eta\) depending only on initial data and the constant \(\rho_0\), satisfying
\begin{align}
C\leq g_{n_0} \leq C^{-1}, \\
C\leq f_{n_0 n_0} \leq C^{-1},
\end{align}
for all \(P=(x,t) \) where \( x \in \Omega(t)\), \(|P-P_0|\leq \eta,\  0\leq t \leq t_0\).
\end{lemma}

\begin{proof}
Let \(\tau(P)\) be the unit vector in the direction of the tangential projection of \(n_0\) and \(\theta\) be the angle between \(n_0\) and the outward normal \(\nu(P)\), for \(P=(x,t) \) where \( x \in \Omega(t)\). Then we get expressions

\begin{align}
g_{n_0}&=\sin \theta g_{\tau} +\cos\theta g_{\nu}, \\
f_{n_0 n_0}&= \cos^2 \theta f_{\nu \nu} +2 \cos \theta \sin \theta f_{\nu \tau} +\sin^2 \theta f_{\tau \tau}.
\end{align}

By Lemma \ref{bound:support},

\begin{equation}
\cos \theta =n_0 \cdot \nu \geq \gamma >0
\end{equation}
for any \(P =(x,t) \) where \( x \in \Omega(t)\) with \(|P-P_0| \leq \eta\). Hence, the desired result follows by the following bounds from Corollary \ref{bound:D^2 f}:

\begin{equation}
C\leq f_{\nu \nu} \ ,\ \frac{1}{g} f_{\tau \tau} \leq C^{-1} \text{   and   } \frac{1}{ \sqrt{ g} } |f_{\tau \nu}|\leq C^{-1}
\end{equation}
provided that \(\eta\) is sufficiently small.

\end{proof}

\begin{lemma}
Let us assume the conditions in the subsection \ref{subsec:assump}. Then there exists a uniform constant \(C>0\) satisfying
\begin{equation}
R_{g,2} \geq -C
\end{equation}
for any sufficiently small value of \(g\) near the interface.
\end{lemma} 

\begin{proof}
To analyze $R_{g, 2}=\sum_{i,j =1} (g_{ii} g_{jj}-g_{ij}^2) =\sum_{i,j \neq 1} (g_{ii} g_{jj}-g_{ij}^2)+2\sum_{i \neq 1}( g_{11} g_{ii}-g_{1i}^2)$, we define a quantity
\begin{equation}
X=\frac{2\sum_{i \neq 1}(g_{11} g_{ii}-g_{1i}^2)}{\sum_{i,j} (g_i^2 g_{jj}+ g_j^2 g_{ii}-2g_i g_j g_{ij} )} +\exp{(b |\nabla g|^2 )}
\end{equation}
on \(\Omega(t), \ 0 \leq t \leq T\).  Getting a uniform lower bound of \( X\), we obtain a uniform lower bound of \( R_{g,2}\). \\

Suppose that \(X\) gets minimum at an interior point \(P_0=P_0(t)\) of \(\Omega(t)\). At \(P_0\), we choose a  coordinate system where \(g_i=0\) for \(i \neq 1\) and \(g_{ij}=0\) for \(i \neq j\). Then 
\begin{equation}
X=\frac{g_{11} }{g_1^2 } +\exp{(b g_1^2 )} , \\
\end{equation}
and we take \(b>0\) sufficiently large that \(X>0 \) at time \(t=0\). We will later determine the value of \(b\) more precisely. \\

The quantity \(X\) evolves in time as
\begin{equation}
\begin{split}
X_t&= \frac{1}{g_1^2 \sum_{k \neq 1} g_{kk}} (g_{11} \sum_{i \neq 1} g_{iit}+\sum_{i \neq 1}g_{ii}g_{11t} ) -\frac{\sum_{i \neq 1}g_{11} g_{ii}}{( g_1^2 \sum_{k \neq 1} g_{kk} )^2} (g_1^2 \sum_{i \neq 1} g_{iit}+2g_1 \sum_{i\neq 1}g_{ii} g_{1t} )  \\
&\ +2b g_1 g_{1t} \exp{(b g_1^2 )} \\
&=\frac{1}{g_1^2 } g_{11t} +\Bigg( 2b g_1\exp{(b g_1^2 )} -\frac{2g_{11} }{g_1^3 } \Bigg) g_{1t} . \\
\end{split}
\end{equation}

If \( g_{11} \geq 0\) at \( P_0\) , then \(X \geq 0\) at \( P_0\) and we are done. So we may assume \( g_{11} \leq 0\) at \( P_0\) . If \( g_{11} \) is bounded below at \( P_0 \), then \(X\) is also bounded below.  Let us assume that \( g_{11} \) is unbounded below.

At \(P_0\),
\begin{equation}
\begin{split}
0=X_1 &=\frac{1}{g_1^2 }g_{111} +\Bigg( -\frac{2\sum_{i \neq 1}g_{11} g_{ii} }{2 g_1^2 (\sum_{k \neq 1} g_{kk} )^2} +\frac{1}{g_1^2 \sum_{k \neq 1} g_{kk}} g_{11} \Bigg) \sum_{i \neq 1} g_{ii1} +\Bigg( 2b g_1 \exp{(b g_1^2 )} -\frac{2\sum_{i \neq 1}g_{11} g_{ii}}{g_1^3 \sum_{k \neq 1} g_{kk} } \Bigg) g_{11} \\
&=\frac{1}{g_1^2 }g_{111} +\Bigg( 2b g_1 \exp{(b g_1^2 )} -\frac{2g_{11} }{g_1^3 } \Bigg) g_{11} 
\end{split}
\end{equation}
and for \(m \neq 1\)
\begin{equation}
\begin{split}
0&=X_m= \frac{1}{g_1^2 }g_{11m} +\Bigg( -\frac{2\sum_{i \neq 1} g_{11} g_{ii} }{2 g_1^2 (\sum_{k \neq 1} g_{kk} )^2} +\frac{1}{g_1^2 \sum_{k \neq 1} g_{kk}} g_{11} \Bigg) \sum_{i \neq 1} g_{iim} = \frac{1}{g_1^2 }g_{11m}  
\end{split}
\end{equation}
so that
\(g_{111}=\Bigg( -2b g_1^3 \exp{(b g_1^2 )} +\frac{2g_{11} }{g_1 } \Bigg) g_{11}\) and \(g_{11m}=0\) for \(m \neq 1\). \\

Accordingly,
\begin{equation}
\begin{split}
X_{11}&=\frac{1}{g_1^2 } g_{1111} +\bigg(2b g_1 \exp{(b g_1^2 )} -\frac{6g_{11} }{g_1^3 } \bigg)\bigg( -2b g_1^3 \exp{(b g_1^2 )} +\frac{2g_{11} }{g_1 } \bigg) g_{11} +\bigg( (4b^2 g_1^2 +2b )\exp{(b g_1^2 )} +\frac{6g_{11} }{g_1^4} \bigg) g_{11}^2 ,\\ 
\end{split}
\end{equation}

and for \(m \neq 1\)
\begin{equation}
\begin{split}
X_{mm}&=\frac{1}{g_1^2 } g_{mm11} -\frac{2}{g_1^2 \sum_{k \neq 1} g_{kk}} \sum_{i \neq 1} g_{im1}^2  \\
&\ +2b g_1 g_{mm1} \exp{(b g_1^2 )} -\frac{2g_{11} }{g_1^3 } g_{mm1} +\frac{4g_{11} }{g_1^3 \sum_{k \neq 1} g_{kk}} g_{mm}g_{mm1}  \\
&\ +2b g_{mm}^2 \exp{(b g_1^2 )} -\frac{2g_{11} }{g_1^4 \sum_{k \neq 1} g_{kk}} ( g_{11}g_{mm}^2 +\sum_{i\neq1} g_{ii}g_{mm}^2 -g_{mm}^3  ) . \\  
\end{split}
\end{equation}

Let us define an operator
\begin{equation}
\begin{split}
LX&:=X_t-\frac{1}{I^{3/2}}g \sum_{i \neq 1} g_{ii} X_{11} -\frac{1}{I^{3/2}} \sum_{m\neq 1} ( Ig \sum_{j \neq 1,m} g_{jj} +g g_{11} +g_1^2 ) X_{mm} \\
& -\frac{2}{I^{3/2}} g \sum_{i \neq 1} g_{ii1} X_1 -\frac{3}{I^{3/2}} g_1 \sum_{j\neq 1} g_{jj}  X_1  \\
& -\frac{1}{2I^{5/2}}\Big(  -Ig^2 \big( gg_1 R_{g,2} +4gg_1 g_{11}\sum_{j \neq 1} g_{jj} \big) -6g^2 g_1^3 \sum_{j \neq 1} g_{jj} +6g^5 g_1^3 g_{11} \sum_{j \neq 1} g_{jj} \Big) X_1  \\
& +\frac{3}{I^{5/2}} ( Ig \sum_{j \neq 1} g_{jj} +g g_{11} +g_1^2 ) g^2 g_1 g_{11} X_1 -\frac{3}{I^{5/2}} g^3 g_1 g_{11}^2 X_1 -\frac{3}{I^{5/2}} ( g^3 g_1^2 \sum_{ j \neq 1} g_{jj} +g_1^2) g^2 g_1 g_{11} X_1 \\ 
& +\frac{3}{I^{5/2}} ( Ig \sum_{j \neq 1} g_{jj} +g g_{11} +g_1^2 ) g g_1^3 X_1 -\frac{1}{I^{3/2}} (2g^2g_1^2 \sum_{j \neq 1} g_{jj} + I \sum_{ j\neq 1} g_{jj} + g_{11}) g_1 X_1 \\
& -\frac{3}{I^{5/2}} g^2 g_1^3 g_{11} X_1 -\frac{3}{I^{5/2}} ( g^3 g_1^2 \sum_{j \neq 1} g_{jj} +g_1^2 ) gg_1^2 g_1 X_1  +\frac{1}{I^{3/2}} (2g^2g_1^2  +I) g_1 g_{11} X_1 \\
& +\frac{3}{I^{3/2}} g^2 g_1^3 \sum_{j \neq 1} g_{jj} X_1 -\frac{3}{I^{3/2}} g^2 g_1^3 g_{11} X_1 +\frac{3}{I^{5/2}} g^3 g_1 g_{11} \sum_{j \neq 1} g_{jj} X_1 +\frac{3}{I^{5/2}} g^2 g_1^3 \sum_{j\neq 1} g_{jj} X_1  \\
& -\bigg( 2b g_1^3 \exp{(b g_1^2 )} -\frac{2g_{11}}{g_1 } \bigg) \frac{1}{I^{3/2} } g \sum_{j \neq 1} g_{jj} X_1 +\frac{2}{I^{3/2}} g g_1^2 \sum_{n \neq 1} X_m^2  . 
\end{split}
\end{equation}

At the minimum point \(P_0(t)\) of \(X\), \(X_{11} \geq 0\) and \(X_{mm} \geq 0\) for all \(m \neq 1\) so that
\begin{equation}
\frac{1}{I^{3/2}}g \sum_{i \neq 1} g_{ii} X_{11} +\frac{1}{I^{3/2}} \sum_{m\neq 1} ( Ig \sum_{j \neq 1,m} g_{jj} +g g_{11} +g_1^2 ) X_{mm} \geq 0
\end{equation}
and \(X_i=0\) for all \(i\) at the minimum point \(P_0=P_0(t)\) of \(X\). So \(X_t \geq LX\) at \(P_0\). Note that, since \(g_{11} \leq 0\) and \(R_{g,2} \leq 0\), the quantity \( P\) which is defined by the formula
\begin{align*}
P&= \frac{1}{g_1^2 } \frac{1}{2I^{5/2}} (I R_{g,2}-6g^ 2g_1^2 g_{11}\sum_{j\neq 1} g_{jj} -6gg_1^4 \sum_{j\neq 1} g_{jj} )g_{11}  \\ 
& +\frac{1}{g_1^2 } \frac{1}{2I^{5/2}}g_{11}^2 \big(4I \sum_{j\neq 1}g_{jj} -Ig^3 R_{g,2} -4Ig^3 g_{11} \sum_{j\neq 1}g_{jj} -2( gg_1 R_{g,2} +4gg_1 g_{11}\sum_{j\neq 1}g_{jj} )g^4 g_1  \\
& \ -10g^2 g_1^2 \sum_{j\neq 1} g_{jj} +18g^5 g_1^2 \sum_{j\neq 1}g_{jj}g_{11}  \big) \\ 
& -\frac{5}{2I^{7/2}}\big( 4I g_1  \sum_{j\neq 1} g_{jj} -Ig^2 ( gg_1 R_{g,2} +4gg_1 g_{11} \sum_{j\neq 1} g_{jj} ) -6g^2 g_1^3 \sum_{j\neq 1} g_{jj} +6g^5 g_1^3 g_{11}\sum_{j\neq 1} g_{jj}\big) g g_1 g_{11} \\ 
& +\frac{5}{2I^{7/2}} g_1^2 (-I g R_{g,2} +6g^3 g_1^2 g_{11}\sum_{j\neq 1} g_{jj}  ) +\frac{15}{I^{7/2}} g^2 g_1^6 \sum_{j\neq 1} g_{jj} -\frac{20}{I^{5/2}} g_1^2 \sum_{j\neq 1} g_{jj} g g_{11} +\frac{1}{I^{5/2}} g^2 R_{g,2} g_{11} \\ 
& +\frac{1}{I^{5/2}} \big( -4 g g_1^2 g_{11}\sum_{j\neq 1} g_{jj} +g g_1^2 (\sum_{i \neq 1} g_{ii})^2-g g_1^2 \sum_{i \neq 1} g_{ii}^2 \big) +\frac{1}{I^{3/2}}  b  g_1^2 \big( (\sum_{i \neq 1} g_{ii})^2 -\sum_{i \neq 1} g_{ii}^2)\exp{(b g_1^2 )}  \\ 
& -\frac{5}{2I^{7/2}} \big(I (\sum_{i \neq 1} g_{ii})^2 -I \sum_{i \neq 1} g_{ii}^2 -6g^2 g_1^2 g_{11}\sum_{j\neq 1} g_{jj} \big) g^2 g_{11}  \\ 
&= \frac{1}{g_1^2 } \frac{1}{2I^{5/2}} \big(I R_{g,2} +O(g) \big)g_{11}  +\frac{1}{g_1^2 } \frac{1}{2I^{5/2}}g_{11}^2 \big(4I \sum_{j\neq 1}g_{jj} +O(g)  \big) -\frac{5}{2I^{7/2}}\big( 4I g_1  \sum_{j\neq 1} g_{jj} +O(g) \big) g g_1 g_{11} \\ 
& +\frac{5}{2I^{7/2}} g_1^2 (-I g R_{g,2} +O(g) ) +\frac{1}{I^{5/2}} \big( -4 g g_1^2 g_{11}\sum_{j\neq 1} g_{jj} +O(g) \big) -\frac{5}{2I^{7/2}} \big(I \sum_{i,j \neq 1\;, i \neq j} g_{ii} g_{jj} +O(g) \big) g^2 g_{11}\\ 
&  +\frac{1}{I^{5/2}} g^2 R_{g,2} g_{11} -\frac{20}{I^{5/2}} g_1^2 \sum_{j\neq 1} g_{jj} g g_{11} +\frac{15}{I^{7/2}} g^2 g_1^6 \sum_{j\neq 1} g_{jj}\\ 
& +\frac{1}{I^{3/2}} b  g_1^2 \big( (\sum_{i \neq 1} g_{ii})^2 -\sum_{i \neq 1} g_{ii}^2)\exp{(b g_1^2 )}  
\end{align*}
is nonnegative when \(g\) is sufficiently small. \\

Since we have \(X=\frac{g_{11}} {g_1^2} +\exp{(b g_1^2)} \) and \(X^2=\frac{g_{11}^2} {g_1^4} +\frac{2g_{11}} {g_1^2} \exp{(b g_1^2)} +\exp{(2b g_1^2)}\) at \(P_0\), we get \(g_{11} = {g_1^2}X -{g_1^2}\exp{(b g_1^2)} \) and \(g_{11}^2=g_1^4 X^2 -2g_1^2 \exp{(b g_1^2)} {g_1^2}X +g_1^4 \exp{(2b g_1^2)} \) so that
\begin{align*}
&LX =  \frac{1}{g_1^2 \sum_{k \neq 1} g_{kk} } \frac{1}{I^{3/2}} 2(g g_{11} +g_1^2) \sum_{i,j \neq 1; i \neq j}g_{ij1}^2 +\frac{1}{g_1^2 \sum_{k \neq 1} g_{kk} } \frac{1}{I^{3/2}} \big( (g g_{11} +g_1^2) -Ig \sum_{l \neq 1} g_{ll} \big) \sum_{i \neq 1} g_{ii1}^2 \\
& +\frac{1}{g_1^2 \sum_{k \neq 1} g_{kk} } \frac{1}{I^{3/2}} Ig  \Big( \sum_{i,j \neq 1; i \neq j; } \sum_{l \neq 1,i,j} g_{ll} g_{ij1}^2 +2\sum_{i,j \neq 1;i\neq j} g_{jj} g_{ii1}^2 \Big) \\
& +\frac{1}{4g_1^2  } \frac{1}{\sqrt{I}} g (\sum_{i \neq 1} g_{ii1})^2 +\frac{1}{3g_1^2 \sum_{k \neq 1} g_{kk} } \frac{1}{I^{3/2}} (g g_{11} +g_1^2) \sum_{i \neq 1}g_{ii1}^2 \\
& +\frac{g}{4g_1^2\sqrt{I}} \bigg( \Big(\sum_{i \neq 1} g_{ii1} -\frac{4g^2 g_1 g_{11} }{I} \sum_{j \neq 1} g_{jj} \Big)^2+ \Big(\sum_{i \neq 1} g_{ii1} -\frac{12g_1 }{I^2} f_{11}^2 \Big)^2 + \Big(\sum_{i \neq 1} g_{ii1} -\frac{8b g_1 g_{11} }{I } \exp{(b g_1^2 )}  \Big)^2 \bigg) \\
& +\frac{1}{12g_1^2 \sum_{k \neq 1} g_{kk} } \frac{1}{I^{3/2}} f_{11} \sum_{i \neq 1} \Bigg( \bigg( g_{ii1} +\frac{12}{f_{11}} g^3 g_1 g_{11} \sum_{j \neq 1} g_{jj} g_{ii}  \bigg)^2+\bigg( g_{ii1} +\frac{24g_{11} }{g_1 f_{11}} Ig g_{ii}^2 \bigg)^2 \\
& \ +\bigg( g_{ii1} -\frac{24}{g_1} g_{11}  g_{ii}\bigg)^2+ \bigg( g_{ii1} -\frac{12}{f_{11} } g_1 \sum_{j \neq 1} g_{jj} g_{ii}  \bigg)^2 +\bigg( g_{ii1} -\frac{24g_{11} }{g_1f_{11} } Ig \sum_{j \neq 1} g_{jj} g_{ii} \bigg)^2 \\
& \ +\bigg( g_{ii1} +\frac{12}{f_{11} } g_1 (\sum_{j \neq 1} g_{jj} )^2 \bigg)^2 + \bigg( g_{ii1} +\frac{36}{f_{11}}g_1  g_{11} \sum_{j \neq 1} g_{jj} \bigg)^2 + \bigg( g_{ii1} + \frac{24}{g_1f_{11} } g g_{11}^2 \sum_{j \neq 1} g_{jj} \bigg)^2 \Bigg) +P \\ 
& +\Big( \frac{b }{ I^{5/2} } (2 -10 g^2 g_1^2) g_1^4\sum_{ j\neq 1} g_{jj} +\mathcal{O}(1)\Big) \exp{(2b g_1^2 )} \\ 
&  -\frac{1}{I^{5/2}} \Big( 6b g^3 g_1^8 \exp{(b g_1^2 )} +\mathcal{O}(g) \Big) X^3 +\frac{1}{I^{5/2}} \big( \mathcal{O}(1) +\mathcal{O}(g)\big) X^2\\ 
& +\frac{1}{I^{5/2}} \Big( b^2\mathcal{O}(g) \exp{(4b g_1^2 )} +(b^2 \mathcal{O}(g) +b \mathcal{O}(g^3))\exp{(2b g_1^2 )} +\big( \mathcal{O}(g) +b^2 \mathcal{O}(g^2) +b \mathcal{O}(g^2) \big) \exp{(b g_1^2 )}\Big) X^2 \\ 
& +\frac{1}{I^{5/2}} \big(\mathcal{O}(1) +\mathcal{O}(g) +b \mathcal{O}(g) +b^2 \mathcal{O}(g)\big)\exp{(b g_1^2 )} X \\ 
& +\frac{2}{I^{3/2}} \sum_{i\neq 1} ( Ig \sum_{j \neq 1,i} g_{jj} +f_{11}) \frac{1}{g_1^2 \sum_{k \neq 1} g_{kk}} \sum_{l \neq 1,i} g_{ll}g_{ii}^2  X  +\frac{1}{I^{5/2}} \big( - I^2 R_{g,2} +\mathcal{O}(g) \big) X \\  
& +\frac{6b}{I^{5/2}} g^3 g_1^8 \exp{(4b g_1^2 )}+\frac{1}{I^{5/2}}\Big( \mathcal{O}(g)+b^2 \mathcal{O}(g) +b\mathcal{O}(g)  \Big) \exp{(3b g_1^2 )}  \\ 
& +\frac{1}{I^{5/2}} \big(b\mathcal{O}( g)+b\mathcal{O}( g^2) +6g^2 g_1^4 g_{11} +\mathcal{O}( g^2) \big) \exp{(2b g_1^2 )}  \\ 
& +\frac{1}{I^{5/2}} \Big( -b\big( 2I \sum_{i\neq 1} ( Ig \sum_{j \neq 1,i} g_{jj} +f_{11} ) g_{ii}^2 +6g g_1^4 f_{11} \sum_{ j \neq 1 } g_{jj} \big)+\mathcal{O}(1) +\mathcal{O}(g) \Big) \exp{(b g_1^2 )} \\ 
& -\frac{12}{I^{3/2}} \frac{1}{f_{11} } \sum_{j \neq 1} g_{jj} \sum_{i \neq 1} g_{ii}^2  -\frac{1}{I^{3/2}} \frac{12(n-1)}{f_{11} } (\sum_{j \neq 1} g_{jj})^3 -\frac{3}{I^{5/2}} g_1^4 \sum_{j\neq 1} g_{jj} -\frac{36}{I^{9/2}} g ( g g_{11} +g_1^2 )^4  .\\ 
\end{align*}

Let us take \(b>0\) sufficiently large, so that \(X_0=X(P_0(0), 0)\) is positive. If X is uniformly bounded below, we are done. Otherwise, there exists some time \(t_0>0\) such that \(X(P_0(t_0), t_0)=0\) for the first time. As \(t \to t_0^+ \), we have \(X=X(P_0(t), t) \to 0^+ \). So there exists some time \(t_1 \in (0, t_0)\) such that for \(t \in [t_1, t_0)\), as \(g\) becomes arbitrarily small,
\begin{equation}
\begin{split}
X_t & \geq  \Big( \frac{2b }{ I^{5/2} } g_1^4\sum_{ j\neq 1} g_{jj} +\mathcal{O}(1)\Big) \exp{(2b g_1^2 )} \\ 
& +\frac{1}{I^{5/2}} \Big( \mathcal{O}(g) +b\mathcal{O}(g^2) +b \mathcal{O}(g^3) \exp{(b g_1^2 )} \Big) X^3 +\frac{1}{I^{5/2}} \big( \mathcal{O}(1) +\mathcal{O}(g)\big) X^2 \\ 
& +\frac{1}{I^{5/2}} \Big( b^2\mathcal{O}(g) \exp{(4b g_1^2 )} +(b^2 \mathcal{O}(g) +b \mathcal{O}(g^3))\exp{(2b g_1^2 )} +\big( \mathcal{O}(g) +b^2 \mathcal{O}(g^2) +b \mathcal{O}(g^2) \big) \exp{(b g_1^2 )}\Big) X^2 \\ 
& +\frac{1}{I^{5/2}} \big(\mathcal{O}(1) +\mathcal{O}(g) +b \mathcal{O}(g) +b^2 \mathcal{O}(g)\big)\exp{(b g_1^2 )} X +\frac{1}{I^{5/2}} \mathcal{O}(g) X \\ 
& +\frac{1}{I^{5/2}}\Big( \mathcal{O}(g)+b^2 \mathcal{O}(g) +b\mathcal{O}(g)  \Big) \exp{(3b g_1^2 )} +\frac{1}{I^{5/2}} \big(b\mathcal{O}( g)+b\mathcal{O}( g^2) +\mathcal{O}( g) \big) \exp{(2b g_1^2 )} \\ 
& +\frac{1}{I^{5/2}} \Big( b\mathcal{O}(1)+\mathcal{O}(1) +\mathcal{O}(g) \Big) \exp{(b g_1^2 )} +\mathcal{O}(1) \\ 
& \geq \frac{b }{ I^{5/2} }  g_1^4 \sum_{ j\neq 1} g_{jj} \exp{(2b g_1^2 )}  >0 \\ 
\end{split}
\end{equation}
for sufficiently large \(b>0\), implying that

\begin{equation}
X(X_0(t),t) \geq X_0 
\end{equation}
where \(X_0=(X_0(t_1),t_1) \) is the data at the time \(t_1\). Hence, \(X\) cannot become zero, which is a contradiction. This concludes the proof that a uniform lower bound of \(R_{g,2}\) exists. 
\end{proof}

\begin{corollary}
Let us assume the conditions in the subsection \ref{subsec:assump}. Then there exists a constant \(C>0\) satisfying
\begin{equation}
g_{11} \geq -C
\end{equation}
for sufficiently small value of \(g\) near the interface. \\
\end{corollary}

\begin{proof}

\begin{equation}
2g_{11} \sum_{i \neq 1} g_{ii} =R_{g,2}-\big(\sum_{i\neq 1} g_{ii} \big)^2 +\sum_{i,j\neq 1} g_{ij}^2 +2\sum_{i \neq 1} g_{1i}^2
\end{equation}
and \(\sum_{i\neq 1} g_{ii}\) is uniformly bounded above and below by the inequality (\ref{inequal:g_ii}).
\end{proof}

In conclusion, we found that \( |g R_{g,2}| \) is uniformly bounded above, and that \( |R_{g,2}| \) is also unifornly bounded above if \( R_{g,2} \leq 0 \) additionally. In the coordinate system where \( g_i =0\) for \( i\neq 1\) and \( g_{ij}=0 \) for \( i \neq j\), this implies that \( |g g_{11}| \) is uniformly bounded above, and that \( |g_{11}| \) is also unifornly bounded above if \( g_{11} \leq 0 \) additionally.

\section{Hölder estimates} \label{sec:holder}
\subsection{\(C^{1, \alpha}_s\) estimates}

We use a local coordinate change from 
\begin{equation}
\big(x_1, x_2 \ldots, x_n, g(x_1,x_2, \ldots, x_n, t)\big)
\end{equation}
 to 
\begin{equation}
\big(h(x_{n+1}, x_2, \ldots, x_n, t), x_2, \ldots, x_n, x_{n+1} \big) . \\
\end{equation}

Let us compute the evolution of the function \(h\). As done in Daskalopoulos and Hamilton \cite{daskalopoulos-hamilton99}, the first-order derivatives of \(g\) are give by
\begin{equation}
\begin{split}
&g_t=-\frac{h_t}{h_{n+1} }, \ g_1=\frac{1}{h_{n+1} }, \ g_i=-\frac{h_i}{h_{n+1} } \ \text{for} \ i \neq 1,
\end{split} 
\end{equation}
and the second-order derivatives of \(g\) are written as
\begin{equation}
\begin{split}
&g_{11}=-\frac{1}{h_{n+1}^3 } h_{n+1, n+1} , \ g_{1i}=-\frac{1}{h_{n+1} }\bigg( -\frac{h_i}{h_{n+1}^2}h_{n+1, n+1}+\frac{1}{h_{n+1} } h_{n+1, i}\bigg) \ \text{for} \ i = 2, \ldots, n, \\
&g_{ii}=-\frac{1}{h_{n+1} }\bigg( \frac{h_i^2 }{h_{n+1}^2}h_{n+1, n+1}-2\frac{h_i}{h_{n+1} } h_{n+1, i}+h_{ii} \bigg) \ \text{for} \ i = 2, \ldots, n, \\
&g_{ij}=-\frac{1}{h_{n+1} }\bigg( \frac{h_i h_j}{h_{n+1}^2}h_{n+1, n+1}-\frac{h_i}{h_{n+1} } h_{n+1, j}-\frac{h_j}{h_{n+1} } h_{n+1, i}+h_{ij} \bigg) \ \text{for} \ i,j = 2, \ldots, n . \\
\end{split}
\end{equation}

Using the notation 
\begin{equation}
\mathcal{I}=h_{n+1}^2 +x_{n+1}^2 +x_{n+1}^2 \sum_{i=2}^n h_i^2 ,
\end{equation} 
the evolution equation of \(h\) is
\begin{equation} \label{eq:h_t}
\begin{split}
&h_t = -\frac{1}{2\sqrt{\mathcal{I}}} h_{n+1}^2 x_{n+1} \Bigg( \frac{1}{h_{n+1}^4 } h_{n+1, n+1} \sum_{i,j=2}^n (h_i^2 h_{jj} +h_j^2 h_{ii} -2h_i h_j h_{ij} ) \\
& \ \ -\frac{1}{h_{n+1}^4 } \sum_{i,j=2}^n ( h_i^2 h_{n+1, j}^2 +h_j^2 h_{n+1, i}^2 -2h_i h_j h_{n+1, i} h_{n+1, j} ) +\frac{2}{h_{n+1}^4 } \sum_{i=2}^n (h_{n+1, n+1}  h_{ii} -h_{n+1, i}^2 ) \\
& \ \ -\frac{2}{h_{n+1}^3 } \sum_{i,j=2}^n (h_i h_{n+1, i} h_{jj} + h_j h_{n+1, j} h_{ii} - h_i h_{n+1, j}h_{ij} -h_j h_{n+1, i}h_{ij} ) +\frac{1}{h_{n+1}^2 } \sum_{i,j=2}^n (h_{ii}h_{jj} -h_{ij}^2 ) \Bigg) \\
&\ +\frac{1}{\sqrt{\mathcal{I}}} h_{n+1}^2 \Bigg( \sum_{i,j=2}^n \frac{1}{h_{n+1}^3 } (h_j^2 h_{ii} - h_i h_j h_{ij}) +\sum_{i=2}^n \frac{1}{h_{n+1}^3 }h_{ii} \Bigg) \\
&\ +\frac{1}{\mathcal{I}^{3/2}} h_{n+1}^4  x_{n+1}^2 \Bigg( -\frac{1}{h_{n+1}^5 } \sum_{i=2}^n  h_{ii} -\frac{1}{h_{n+1}^5 } \sum_{i,j=2}^n h_i^2 h_{jj} \\
&\ \ +\frac{x_{n+1}}{h_{n+1}^6 } \sum_{i=2}^n (h_{ii} h_{n+1, n+1} - h_{n+1, i}^2 ) +\frac{x_{n+1} }{h_{n+1}^6 } \sum_{i,j=2}^n h_i h_j ( h_{ij} h_{n+1, n+1} -h_{n+1, i} h_{n+1, j} ) \\
&\ \ -\frac{x_{n+1} }{h_{n+1}^5 } \sum_{i,j,k=2}^n h_i h_k \big( h_{jj} ( h_i h_{n+1, k}+h_k h_{n+1, i} ) -h_{jk} ( h_i h_{n+1, j}+h_j h_{n+1, i} ) \big) \\
&\ \ -\frac{1}{h_{n+1}^5 } \sum_{i,j=2}^n ( h_i^2 h_{jj}  -h_i h_j h_{ij} ) +\frac{1}{h_{n+1}^5 } \sum_{i,j,k=2}^n h_k^2 (h_i h_j h_{ij} -h_i^2 h_{jj}) \\
&\ \ +\frac{2x_{n+1} }{h_{n+1}^6 } \sum_{i,j=2}^n (h_i h_j h_{n+1, i} h_{n+1, j}  -h_i^2 h_{n+1, j}^2) +\frac{2x_{n+1} }{h_{n+1}^6 } h_{n+1, n+1} \sum_{i,j=2}^n (h_i^2 h_{jj} -h_i h_j h_{ij} )  \\
&\ \ +\frac{x_{n+1} }{h_{n+1}^6 } \sum_{i,j,k=2}^n h_k^2 (h_i h_j h_{n+1, i} h_{n+1, j} -h_i^2 h_{n+1, j}^2 ) +\frac{x_{n+1} }{h_{n+1}^6 } h_{n+1, n+1} \sum_{i,j,k=2}^n h_k^2 ( h_i^2 h_{jj} -h_i h_j h_{ij} )\\
&\ \ +\frac{2x_{n+1} }{h_{n+1}^5 } \sum_{i,j=2}^n (h_i h_{n+1, j} h_{ij} -h_i h_{n+1, i} h_{jj} ) +\frac{x_{n+1} }{h_{n+1}^4 } \sum_{i,j,k=2}^n h_i h_k (h_{jj} h_{ik} - h_{ij} h_{jk} ) \Bigg) . \\
\end{split}
\end{equation}

If we use a local coordinate system where \(g_i=0\) for \(i =2, \ldots , n\) and \(g_{ij}\) for \(i \neq j\), then 
\begin{equation} \label{eq1: Dg in h}
\begin{split}
h_{n+1}=\frac{1}{g_1} \geq  0, \ h_i=0 \ \text{for} \ i =2, \ldots , n,
\end{split} 
\end{equation}
and the second-order derivatives of \(g\) are written as
\begin{equation}  \label{eq2: Dg in h}
\begin{split}
&h_{n+1, n+1} = - h_{n+1}^3 g_{11} , \ h_{n+1, i}=0 \ \text{for} \ i = 2, \ldots, n, \\
& h_{ii}=-h_{n+1} g_{ii} \leq 0 \ \text{for} \ i = 2, \ldots, n, \ h_{ij}=0  \ \text{for} \ i,j = 2, \ldots, n \text{ such that } i \neq j. 
\end{split}
\end{equation}

Consequently,
\begin{equation}  \label{eq3: Dg in h}
\begin{split}
C^{-1} \leq & h_{n+1} \leq C,  \\
0 \leq -x_{n+1} h_{n+1, n+1} = h_{n+1}^3 g g_{11} \leq C &\text{ or } -gC \leq -x_{n+1} h_{n+1, n+1} = h_{n+1}^3 g g_{11} \leq 0, \\
-C \leq & \sum_{i=2}^n h_{ii} \leq -C^{-1},  \\
\end{split}
\end{equation}

\begin{equation}  \label{eq4: Dg in h}
\begin{split}
-C &\leq R_{g,2} =\frac{1}{h_{n+1}^2 }\sum_{i,j=2}^n (h_{ii}h_{jj} -h_{ij}^2 ) +\frac{2}{h_{n+1}^4 } \sum_{i=2}^n h_{n+1, n+1} h_{ii} \leq 0,  \text{ or } \\
0 &\leq g R_{g,2} =\frac{x_{n+1} }{h_{n+1}^2 }\sum_{i,j=2}^n (h_{ii}h_{jj} -h_{ij}^2 ) +\frac{2x_{n+1} }{h_{n+1}^4 } \sum_{i=2}^n h_{n+1, n+1} h_{ii} \leq C, \\
\end{split}
\end{equation}

and
\begin{equation}  \label{eq5: Dg in h}
C^{-1} \leq  g_1^2 \sum_{i=2}^n g_{ii} = -\frac{1}{h_{n+1}^3 }\sum_{i=2}^n h_{ii} \leq C 
\end{equation}
for a uniform constant \(C>0\), from Lemma \ref{upperbound:Dg}, \ref{lowerbound:Dg}, and Corollary \ref{bound:g_ii}. \\

From the relations (\ref{eq1: Dg in h})-(\ref{eq5: Dg in h}) and straightforward calculations, the evolution equation of \(\widetilde{h}\), which is either \(h_t\) or \(h_i\), \(i=2, \ldots , n\), is the linearized equation
\begin{equation} \label{eq:tildeh_t}
\begin{split}
\widetilde{h}_t  &=  x_{n+1} \Bigg( -\frac{1}{\sqrt{\mathcal{I}}} \frac{1}{h_{n+1}^2 } \sum_{i=2}^n h_{ii} +\mathcal{O}(x_{n+1}^2) \Bigg)\widetilde{h}_{n+1, n+1}  \\ 
&\ +\Bigg( \frac{1}{\sqrt{\mathcal{I}}} \frac{1}{h_{n+1}} -\frac{1}{\sqrt{\mathcal{I}}} \frac{1}{h_{n+1}^2 } x_{n+1} h_{n+1, n+1} +\mathcal{O}(x_{n+1}) \Bigg) \sum_{i=2}^n \widetilde{h}_{ii} \\
&\ +\Bigg( -\frac{1}{\sqrt{\mathcal{I}}} \frac{1}{h_{n+1}^2}\sum_{i=2}^n h_{ii} +\mathcal{O}(x_{n+1}) \Bigg) \widetilde{h}_{n+1} . \\ 
\end{split}
\end{equation}

It is of the form
\begin{equation}
\widetilde{h}_t  =  x_{n+1} a_{n+1, n+1} \widetilde{h}_{n+1, n+1}  +\sum_{i=2}^n a_{ii} \widetilde{h}_{ii}  +b_{n+1} \widetilde{h}_{n+1}  . \\ 
\end{equation}

We immediately see that the matrix \((a_{ij})\) is uniformly elliptic and strictly positive, and \(b_{n+1} >0\) is uniformly bounded, for sufficiently small \(g=x_{n+1}\). Thus, with minor changes in the higher dimension to the line of proof of Theorem 3.1 in \cite{daskalopoulos-lee03}, we obtain the following lemma.
\begin{lemma}
There exist a number \(0<\alpha<1\) so that, for any \(r< \rho\)
\begin{equation}
\Vert h_t \Vert_{C^\alpha_s (\mathcal{B}_r)} +\sum_{i=2}^n \Vert h_i \Vert_{C^\alpha_s (\mathcal{B}_r)} \leq C(r, \rho) \Vert \widetilde{h} \Vert _{C^0 (\mathcal{B}_1)}
\end{equation}
with respect to the singular distance function
\begin{equation}
s( (x_1, t_1), (x_2, t_2))= |\sqrt{x_{1,n+1}}-\sqrt{x_{2,n+1}}|+|(x_{1,2}, \ldots, x_{1,n})-(x_{2,2}, \ldots, x_{2,n})|+\sqrt{|t_1-t_2|}
\end{equation}
where \(\mathcal{B}_\rho=\mathcal{B}_\rho(x_0,t_0) \) is the parabolic box
\begin{equation}
\begin{split}
\{(x_{n+1}, x_2, \ldots, x_n, t) : &\ x_{n+1} \geq 0, |x_{n+1}-x_{0, n+1}|\leq \rho^2, |(x_2, \ldots, x_n)-(x_{0,2}, \ldots, x_{0,n})|\leq \rho, \\
& -\rho^2 \leq t -t_0 \leq 0  \} .
\end{split}
\end{equation}
\end{lemma}

Consequently,  \(h_t\), and \(h_i\), \(i=2, \ldots , n\) belong to \(C^{\alpha}_s (\mathcal{B}_\frac{\rho}{2})\) near the free boundary.

For \(\widetilde{h}=h_{n+1}\), straightforward calculations show that in a local coordinates and the relations (\ref{eq1: Dg in h})-(\ref{eq5: Dg in h}) with the choice of coordinates \(g_i=0\) for \(i =2, \ldots , n\) and \(g_{ij}\) for \(i \neq j\), 
\begin{equation}
\begin{split}
\widetilde{h}_t  &=  x_{n+1} \Bigg( -\frac{1}{\sqrt{\mathcal{I}}} \frac{1}{h_{n+1}^2 } \sum_{i=2}^n h_{ii} +\mathcal{O}(x_{n+1}^2) \Bigg)\widetilde{h}_{n+1, n+1} \\
&\ +\Bigg( \frac{1}{\sqrt{\mathcal{I}}} \frac{1}{h_{n+1}} -\frac{1}{\sqrt{\mathcal{I}}} \frac{1}{h_{n+1}^2 } x_{n+1} h_{n+1, n+1} +\mathcal{O}(x_{n+1}) \Bigg) \sum_{i=2}^n \widetilde{h}_{ii}   \\ 
&\ +\Bigg( -\frac{2}{\sqrt{\mathcal{I}}} \frac{1}{h_{n+1}^2}\sum_{i=2}^n h_{ii} +\mathcal{O}(x_{n+1}) \Bigg) \widetilde{h}_{n+1}-\frac{1}{2\sqrt{\mathcal{I}}} (\sum_{i,j=2}^n h_{ii}h_{jj} -\sum_{i=2}^n h_{ii}^2 ) +\mathcal{O}(x_{n+1}) . \\ 
\end{split}
\end{equation}

Since the part \( -\frac{1}{2\sqrt{\mathcal{I}}} (\sum_{i,j=2}^n h_{ii}h_{jj} -\sum_{i=2}^n h_{ii}^2 )  =-\frac{1}{2\sqrt{\mathcal{I}}} h_{n+1}^2 (\sum_{i,j=2}^n g_{ii} g_{jj} -\sum_{i=2}^n g_{ii}^2 ) \) is uniformly bounded for small \(x_{n+1}\), we can apply the same estimate as \(h_t\) and \(h_i\), \(i =2, \ldots, n\).

\begin{lemma} There exist a number \(0<\alpha<1\) and positive constants \(\eta>0\) and \(C>0\) depending only on the initial data and \(\rho_0\) such that
\begin{equation}
\Vert {h_t} \Vert_{C^\alpha_s (\mathcal{B}_\eta)} \leq C , \  \Vert {h_{n+1} } \Vert_{C^\alpha_s (\mathcal{B}_\eta)} \leq C \text{ and } \Vert {h_i} \Vert_{C^\alpha_s (\mathcal{B}_\eta)} \leq C \text{ for } i=2, \ldots , n.
\end{equation}
\end{lemma}

Thus, we see that \(h \in C^{1,\alpha}_s (\mathcal{B}_\eta)\).

\subsection{\(C^{2, \alpha}_s\) estimates}
For the point \(P=(z_{n+1}, z_2, \ldots, z_n)\) on the free boundary \(\Gamma(t)\), let us denote \(z=z_{n+1}\) and \(y=(z_2, \ldots, z_n)\). Let \(C^{2,\alpha}_s (\mathcal{B}_\eta )\) be the space of functions on \(\mathcal{B}_\eta \)
\begin{equation}
C^{2,\alpha}_s (\mathcal{B}_\eta )=\{h \ \vert \ h , h_t, h_{n+1}, h_i, x_{n+1} h_{n+1, n+1},  \sqrt{x_{n+1}} h_{n+1, i}, h_{ij} \in C^{\alpha}_s (\mathcal{C}_\eta), i,j =2, \ldots n \},
\end{equation}
where the parabolic box \(\mathcal{B}_\eta\) is centered at the point \((0,y_0, t_0) \in \Gamma(t_0)\).
Now, we want to get the \(C^{2,\alpha}_s\) regularity of \(h\) on \(\mathcal{B}_\eta\) from its \(C^{1,\alpha}_s\) regularity and the classical regularity theory for strictly parabolic equations, as done for the Gauss curvature flow in \(\cite{daskalopoulos-lee04}\). \\

For \(0<\mu<1\), let us \(C_\mu\) denote the parabolic cylinder
\begin{equation}
\mathcal{C}_\mu=\{ z^2+|y|^2 \leq \mu^2, -\mu^2 \leq t  \leq 0 \}.
\end{equation}

Let \(h^r\) be the dilated function of \(h\) at a point \(Q^r=(r^2, y_r, t_r) \in \mathcal{B}_\eta\)
\begin{equation}
h^r (z,y,t)=\frac{1}{r^2} h(r^2+r^2 z, y_r +r y, t_r +r^2 t).
\end{equation}

Then the evolution of \(h^r\) is as follows.
\begin{lemma}
In our coordinate where \(g_i=0\) for \(i =2 , \ldots ,n\) and \(g_{ij}=0\) for \(i \neq j\), 
\begin{equation}
\begin{split}
h^r_t &=  -\frac{\tilde{z} }{ (h^r_{n+1})^2 \sqrt{\mathcal{I}} } \bigg(1-\frac{ (r^2 \tilde{z})^2}{\mathcal{I}} \bigg) \big( \sum_{i=2}^n h^r_{ii} \big) h^r_{n+1, n+1}   +\frac{1}{h^r_{n+1}\sqrt{\mathcal{I}} } \bigg(1-\frac{ (r^2 \tilde{z})^2}{\mathcal{I}} \bigg)\sum_{i=2}^n h^r_{ii}  \\ 
&\ -\frac{r^2 \tilde{z} }{2\sqrt{\mathcal{I}}} \sum_{i,j=2}^n \big(h^r_{ii}h^r_{jj} -(h^r_{ij})^2 \big) . \\
\end{split}
\end{equation}

\end{lemma}

\begin{proof}
Let \(\tilde{z}=1+z\). Then
\begin{equation}
\mathcal{I}_{\vert(r^2+r^2 z, y_r +r y, t_r +r^2 t)} = (h^r_{n+1})^2 + r^4 (\tilde{z})^2 +r^6 (\tilde{z})^2 \sum_{i=2}^n (h^r_i)^2  \\
\end{equation}
is uniformly bounded below and above. And for \(i,j =2, \ldots, n\), at \((r^2+r^2 z, y_r +r y, t_r +r^2 t)\), we have
\begin{equation}
h_{n+1} = h^r_{n+1},\ h_i =r h^r_i , \ h_{n+1,n+1} =\frac{1}{r^2} h^r_{n+1,n+1} ,\ h_{n+1,i} =\frac{1}{r} h^r_{n+1,i} ,\ h_{ij} = h^r_{ij} .\\
\end{equation}
from which we get the result.
\end{proof}

\begin{lemma}
For any \(0<\mu_0<1\), there exists a constant \(C>0\) depending on \(\mu_0\), \(\rho_0\) and the initial data such that
\begin{equation}
\Vert h^r \Vert_{C^\infty (\mathcal{C}_\mu)} \leq C
\end{equation}
for all \(0<\mu <\mu_0\).
\end{lemma}

\begin{proof}
If \(P=(z, y,t) \in \mathcal{C}_\mu\) with \(0 < \mu <1\), then \(\tilde{z} \geq 1-\mu^2 >0\). As the derivatives of \(h\) are uniformly bounded, by the relation between the derivatives of \(h^r\) and \(h\) above, \(h^r\) satisfies a uniformly parabolic equation with ellipticity constant independent of \(r\). \\

Hence, from the regularity of solutions to fully nonlinear uniformly parabolic equations (see Wang \cite{wang92a} and Wang \cite{wang92b}), \(\Vert h^r \Vert_{C^\infty (\mathcal{C}_\mu) }\) is, up to a uniform constant, bounded by \(\Vert h^r \Vert_{L^\infty (\mathcal{C}_{\mu_0})}\)  for all \(0<\mu<\mu_0<1\). Since \(h_{n+1}\) is bounded in \(\mathcal{B}_{\eta}\), \(\Vert h^r \Vert_{L^\infty (\mathcal{C}_{\mu_0})}\) is uniformly bounded. 
\end{proof}

Hence Lemma 6.4, Lemma 6.5 and Lemma 6.6 in Daskalopoulos and Lee \(\cite{daskalopoulos-lee03}\) also hold for our Gauss curvature flow with a flat side as well, and we get the following lemma.
\begin{lemma}
There exists a constant \(C>0\), depending only on the initial data, \(\rho_0\) and \(\eta\), such that for any two points \(P_1=(z_1, y_1, t_1)\) and \(P_2=(z_2, y_2, t_2)\) in \(\mathcal{B}_{\frac{\eta}{2} }\), we have
\begin{equation}
|z_1 h_{n+1, n+1}(P_1)-z_2 h_{n+1,n+1}(P_2)|+\sum_{i=2}^n |\sqrt{z_1} h_{n+1, i}(P_1)-\sqrt{z_2} h_{n+1,i}(P_2)| \leq C s(P_1, P_2)^\alpha
\end{equation}
for our metric distance \(s(P_1, P_2)=\sqrt{|z_1-z_2|}+|y_1-y_2|+\sqrt{|t_1-t_2|} \). In other words, 
\begin{equation}
x_{n+1} h_{n+1, n+1} \in \mathcal{C}^\alpha (\mathcal{B}_{\frac{\eta}{2}} ),  \sqrt{x_{n+1} } h_{n+1, i} \in \mathcal{C}^\alpha (\mathcal{B}_{\frac{\eta}{2}} ) \text{ for } i=2, \ldots , n. 
\end{equation}
\\
\end{lemma}

Also, we get the Hölder regularity of \(h_{ii}\), \(i=2, \ldots, n\). 
\begin{lemma}
\begin{equation}
h_{ii} \in \mathcal{C}^\alpha (\mathcal{B}_{\frac{\eta}{2}} ) \text{ for } i=2, \ldots , n.
\end{equation}
\end{lemma}

\begin{proof}
In our coordinate where \(g_i=0\) for \(i =2 , \ldots ,n\) and \(g_{ij}=0\) for \(i \neq j\), 
\begin{equation}
\begin{split}
&h_t = -\frac{ x_{n+1}}{2\sqrt{\mathcal{I}}} \big( \sum_{i,j=2}^n h_{ii}h_{jj} -\sum_{i=2}^n h_{ii}^2 \big) +\frac{1}{\sqrt{\mathcal{I}}} \bigg( 1 -\frac{  x_{n+1}^2}{\mathcal{I}} \bigg) \bigg( \frac{1}{h_{n+1} }  - \frac{x_{n+1}}{h_{n+1}^2 } h_{n+1, n+1}  \bigg) \sum_{i=2}^n h_{ii} \\
&= -\frac{ x_{n+1}}{2\sqrt{\mathcal{I}}} \big( \sum_{i,j=2 ; i \neq j}^n h_{ii}h_{jj} \big) +\frac{1}{ h_{n+1}^2 \mathcal{I}^{3/2} } (h_{n+1}^2 +x_{n+1}^2  \sum_{i=2}^n h_i^2 ) ( h_{n+1}  - x_{n+1} h_{n+1, n+1} ) \sum_{i=2}^n h_{ii} \\
\end{split}
\end{equation}

so that
\begin{equation}
\sum_{i=2}^n h_{ii} =\frac{\mathcal{I}^{3/2} h_t +\mathcal{I} \mathcal{O}(x_{n+1}) }{ h_{n+1} +\mathcal{O}(x_{n+1} )} .
\end{equation}

Near the free boundary, \(x_{n+1}=g\) tends to zero. Note that the denominator \(h_{n+1} +\mathcal{O}(x_{n+1}) \geq C >0\) on \(\mathcal{B}_\eta \) for some uniform constant \(C>0\), from the estimate of \(h_{n+1}\). Also, the numerator \(\mathcal{I}^{3/2} h_t +\mathcal{I} \mathcal{O}(x_{n+1}) \in \mathcal{C}^\alpha (\mathcal{B}_{\frac{\eta}{2}} )\), by the Hölder regularity of \(h_t\), \(h_{n+1}\) and \(h_i\), \(i=2,\ldots, n\).
\end{proof}

Consequently, the Hölder regularity of the derivatives of \(h\) of order one and of order two implies the following theorem.
\begin{theorem} \label{thm:C2a}
Assume that \(g=\sqrt{2f} \) is of class \(\mathcal{C}^{2+\beta} \) up to the interface \(z=0\) for some \(0<\beta<1\), and satisfies the non-degeneracy condition 
\begin{equation}
\vert Dg(x) \vert \geq \lambda \text{ and } \sum_{\tau :\ tangential} D^2_{\tau \tau} g(x)  \geq \lambda \text{ at all } x \in \Gamma,
\end{equation}
for some positive number \(\lambda>0\). Also, assume that a ball is included in the flat side: \(D_{\rho_0}={ |x|\leq \rho_0 } \subset \Sigma_0(T)\). Then there exist uniform constants \(0<\alpha<1\), \(0<C<\infty\) and \(\eta>0\) which depend only on the initial data and \(\rho_0\), such that for any free boundary point \(P_0=(x_0, y_0,t_0)\) with \(0<\tau <t_0<T\) satisfying \(n_0 :=\frac{P_0}{|P_0|}=e_1\), the function \(x=h(z,y,t)=h(x_{n+1},x_2, \ldots, x_n,t)\) satisfies the Hölder estimate
\begin{equation}
\Vert h\Vert_{\mathcal{C}^{2+\alpha}_s (\mathcal{B}_\eta)} \leq C
\end{equation}
where \(\mathcal{B}_\eta=\{ 0 \leq z \leq \eta^2, |y-y_0|\leq \eta, t_0-\eta^2 \leq t \leq t_0 \} \). 
\end{theorem}

\section{Conclusion: proof of the main theorem} \label{sec:existence}
\subsection{Short-time existence near the interface}

The last results have an immediate consequence, the short-time existence of the \(h\) near a free boundary point.
\begin{lemma}
There exists a unique solution \(h \in \mathcal{C}^{2+\alpha}_s (\mathcal{B}_\eta)\) of the equation (\ref{eq:h_t}) for a short time \(T>0\) in \(\mathcal{B}_\eta\) for some constant \(\eta>0\) as in Theorem \ref{thm:C2a}.
\end{lemma}
\begin{proof}
The linearized equation (\ref{eq:tildeh_t}) of \(Lh=0\) (\ref{eq:h_t}) satisfies the condition of Theorem 7.1. in \cite{daskalopoulos-hamilton99} with \(k=0\) and minor changes in higher dimension \(n \geq 3\). Applying the Inverse Function Theorem, we get the existence result for \(h\). 
\end{proof}

Thus, we get the short-time existence of \(h\) along the interface.
\begin{theorem} \label{thm:short}
There exists a unique smooth solution \(f\) to the scalar curvature flow for a short time \(T>0\).
\end{theorem}
\begin{proof}
Let us cover the interface \(\Gamma=\Gamma(0)\) into balls, each of which is centered at a free boundary point on it. We use a coordinate in each ball such that the free boundary point is \(0\), the free boundary is flat, and the area inside the flat side is the upper half plane. Then we have the short time existence of \(h\in \mathcal{C}^{2+\alpha}_s (\mathcal{B}_\eta) \) inside each ball. Since \(\Gamma\) is compact, it can be covered by a finite number of such balls as above. The short-time existence of the functions \(g=\sqrt{2f}\) and \(f\) follows.
\end{proof}

\subsection{All-time \(C^{\infty}\) regularity up to the interface}
Now, we prove a lemma to prove the main theorem.

\begin{lemma} \label{lem:holder of g}
Let \(g\) be the solution which is smooth up to the interface on \(0<t<T\) and \(T<T_c\). Then \(g(x)=g(x,T)\) belongs to the class \(\mathcal{C}^{2+\beta}_s\) for some \(0<\beta<1\) and satisfies the non-degenaracy conditions \(\vert D g(x) \vert \geq \lambda\) and  \(D^2_{\tau \tau} g(x) \geq \lambda\) for any \(x \in \Gamma\) for some constant \(\lambda>0\).
\end{lemma}
\begin{proof}
By the theorem \ref{thm:C2a} about the Hölder estimate of \(h\) with the relations (\ref{eq1: Dg in h})-(\ref{eq5: Dg in h}) between the first-order and second-order derivatives of \(g\) and \(h\), the conclusion is immediate.
\end{proof}

Finally, we prove our main theorem \ref{thm:main}.
\begin{proof} { \bf{Proof of Theorem \ref{thm:main} } }
By the short-time existence theorem \ref{thm:short}, there exists a solution \(g\), smooth up to the interface in \(0<t<T\), for a maximal time \(T>0\). If \(T<T_c\), then \(g(\cdot, T)\) belongs to the class \(\mathcal{C}^{2+\beta}_s\) up to the interface \(x_{n+1}=0\) for some \(0<\beta<1\) and it satisfies the degenaracy conditions by the lemma \ref{lem:holder of g}. Then the linearization (\ref{eq:g_m}) of the evolution equation (\ref{eq:g_t}) of \(g\) satisfies the condition of Theorem 7.1 in \cite {daskalopoulos-hamilton99} with minor changes in higer dimension \(n \geq 3\). Applying the Inverse Function Theorem, we get the short-time existence for \(g(x,t)\) with the initial data \(g(x,T)\), for all \(T\leq t <T+T'\) for some \(T'>0\) and it is \(\mathcal{C}^{\infty}\) up to the interface. This is contradiction to the condition that \(T<T_c\) is the maximal time. Hence, we must have \(T=T_c\), the critical time of the flow.
\end{proof}

\section*{\bf{Acknowledgements}}
Hyo Seok Jang was supported by the Lecture \& Research Scholarship funded by Seoul National University, and by the Brain Korea 21 Plus Scholarship funded by the Korea Goverment (MSIP). Ki-Ahm Lee was supported by the National Research Foundation of Korea (NRF) grant funded by the Korea goverment (MSIP) (NO. 2017R1A2A2A05001376).

\end{document}